\documentclass[reqno,11pt]{amsart}
\usepackage{amsmath}
\usepackage{amssymb}
\usepackage{tikz}
\usepackage{datetime2}
\usetikzlibrary{shapes,snakes}

\newtheorem{theorem}{Theorem}[section]
\newtheorem{lemma}[theorem]{Lemma}

\newtheorem{proposition}[theorem]{Proposition}
\newtheorem{definition}[theorem]{Definition}
\newtheorem{remark}[theorem]{Remark}

\newtheorem{corollary}[theorem]{Corollary}

\usepackage{amsmath}
\usepackage{amssymb}
\usepackage{amsfonts}
\usepackage{enumerate}
\usepackage{comment}
\usepackage{braket}

\usepackage{palatino}
\usepackage[T1]{fontenc}
\usepackage[mathscr]{eucal}
\usepackage{acronym}
\usepackage{latexsym}
\usepackage{paralist}
\usepackage{wasysym}
\usepackage{xspace}

\usepackage[font=small,labelfont=bf]{caption}
\usepackage{subfigure}
\usepackage{tikz}
\usetikzlibrary{calc,patterns}

\usepackage{hyperref}
\hypersetup{
colorlinks=true,
linktocpage=true,
pdfstartview=FitH,
breaklinks=true,
pdfpagemode=UseNone,
pageanchor=true,
pdfpagemode=UseOutlines,
plainpages=false,
bookmarksnumbered,
bookmarksopen=false,
bookmarksopenlevel=1,
hypertexnames=true,
pdfhighlight=/O,
urlcolor=blue!60!black,linkcolor=blue!70!black,citecolor=green!70!black, % <--- for screen
pdftitle={},
pdfauthor={},
pdfsubject={},
pdfkeywords={},
pdfcreator={pdfLaTeX},
pdfproducer={LaTeX with hyperref}
}

\numberwithin{equation}{section}  %numberwithin goes before cleverefs when using hyperref
\usepackage[sort&compress,capitalize,nameinlink]{cleveref}
\crefname{app}{Appendix}{Appendices}

\crefrangeformat{equation}{\upshape(#3#1#4)\textendash(#5#2#6)}

\usepackage[textwidth=30mm]{todonotes}

\usepackage{soul}
\setstcolor{red}
\sethlcolor{SkyBlue}

\newcommand{\w}{\mathbf{w}}

\newcommand{\si}{\sigma} 
\renewcommand{\l}{\lambda} 
 
\newcommand{\e}{\varepsilon} 
\renewcommand{\a}{\alpha}

\newcommand{\ent}{{\rm ENT} } 
\newcommand{\var}{{\rm Var} }

\newcommand{\ind}{\mathbf{1}}
\newcommand{\p}{\mathfrak{p}}

\newcommand{\bd}{\mathbf d}

\newcommand{\D}{\Delta}
\renewcommand{\d}{\delta}

\renewcommand{\t}{\tau}

\newcommand{\tv}{\texttt{TV}}
\newcommand{\tent}{T_\ent}
\newcommand{\muin}{\mu_{{\rm in}}}

\newcommand{\cC}{\ensuremath{\mathcal C}} 
\newcommand{\cD}{\ensuremath{\mathcal D}} 
\newcommand{\cE}{\ensuremath{\mathcal E}} 
\newcommand{\cF}{\ensuremath{\mathcal F}} 
\newcommand{\cG}{\ensuremath{\mathcal G}} 
 
\newcommand{\cI}{\ensuremath{\mathcal I}} 
\newcommand{\cJ}{\ensuremath{\mathcal J}} 
 
\newcommand{\cL}{\ensuremath{\mathcal L}} 
 
\newcommand{\cN}{\ensuremath{\mathcal N}} 
 
\newcommand{\cP}{\ensuremath{\mathcal P}}

\newcommand{\cS}{\ensuremath{\mathcal S}} 
\newcommand{\cT}{\ensuremath{\mathcal T}}

\newcommand{\cW}{\ensuremath{\mathcal W}} 
 
\newcommand{\cY}{\ensuremath{\mathcal Y}}

\newcommand{\bbE}{{\ensuremath{\mathbb E}} }

\newcommand{\bbN}{{\ensuremath{\mathbb N}} } 
 
\newcommand{\bbP}{{\ensuremath{\mathbb P}} }

\newcommand{\bu}{\mathbf{u}}
\newcommand{\Pan}{\mathbf{P}^{\rm an}} %{\mathbf{an}}}
\newcommand{\Pj}{\mathbf{P}^{\rm J}} %{\mathbf{J}}}
\newcommand{\E}{\ensuremath{\mathbb{E}}}

\renewcommand{\P}{\ensuremath{\mathbb{P}}}

\def\({\left(}
\def\){\right)}
\def\[{\left[}
\def\]{\right]}
\newacro{NE}{Nash equilibrium}
\newacroplural{NE}[NE]{Nash equilibria}
\newacro{PNE}{pure Nash equilibrium}
\newacroplural{PNE}[PNE]{Nash equilibria}
\newacro{PFNE}{prior-free Nash equilibrium}
\newacroplural{PFNE}[PFNE]{prior-free Nash equilibria}
\newacro{WE}{Wardrop equilibrium}
\newacroplural{WE}[WE]{Wardrop equilibria}
\newacro{SO}{socially optimum}

\newacro{KKT}{Karush\textendash Kuhn\textendash Tucker}
\newacro{OD}[O/D]{origin-destination}
\newacro{PoA}{price of anarchy}
\newacro{PoS}{price of stability}
\newacro{PoCS}{price of correlated stability}
\newacro{BPR}{bureau of public roads}
\newacro{FIP}{finite improvement property}

\newacro{BPG}{buck-passing game}
\newacro{SBPG}{stochastic buck-passing game}
\newacro{MBPG}{mixed extension of the buck-passing game}

\begin{document}
\title[Mixing time trichotomy in regenerating dynamic digraphs]{Mixing time trichotomy in regenerating dynamic digraphs}
\author[P.~Caputo]{Pietro Caputo$^{\#}$}
\address{$^{\#}$ Dipartimento di Matematica e Fisica, Universit\`a di Roma Tre, Largo S. Leonardo Murialdo 1, 00146 Roma, Italy.}
\email{caputo@mat.uniroma3.it}

\author[M.~Quattropani]{Matteo Quattropani$^{\flat}$}
\address{$^{\flat}$ Dipartimento di Matematica e Fisica, Universit\`a di Roma Tre, Largo S. Leonardo Murialdo 1, 00146 Roma, Italy.}
\email{matteo.quattropani@uniroma3.it}
%\urladdr{\url{http://www.there.com}}

%
\begin{abstract}
We study the convergence to stationarity for random walks on  dynamic random digraphs with given degree sequences. The digraphs undergo full regeneration at independent geometrically distributed random time intervals with parameter $\a$. Relaxation to stationarity is the result of a competition between regeneration and mixing on the static digraph. When the number of vertices $n$ tends to infinity and the parameter $\a$ tends to zero, we find three scenarios according to whether $\a\log n$ converges to zero, infinity or to some finite positive value: when the limit is zero, relaxation to stationarity occurs in two separate stages, the first due to  mixing on the static digraph, and the second due to regeneration; when the limit is infinite, there is not enough time for the static digraph to mix and the relaxation to stationarity is dictated by the regeneration only; finally, when the limit is a finite positive value we find a mixed behaviour interpolating between the two extremes.  A crucial ingredient of our analysis is the control of suitable approximations for the unknown stationary distribution. 
\end{abstract}


\maketitle             % typeset the header of the contribution

%\tableofcontents
%
% Section ------------------------------------------
%
%\tableofcontents
\section{Introduction}
The analysis of stochastic processes on dynamic networks constitutes a fundamental theme of current and future research \cite{michail2018elements}. In this paper we are interested in the mixing time of random walks on dynamic graphs. In contrast with the case of static graphs, the theory for graphs evolving with time is far from being fully developed. 
Related problems have been considered in the literature on the so-called random walks in dynamic random environments; see e.g.\ \cite{boldrighini1997almost,dolgopyat2008random,bandyopadhyay2005random,avena2011law,figueiredo2012characterizing} and references therein. More closely linked to our questions here is the analysis of mixing times of random walks on dynamical percolation \cite{peres2015random,sousi2018cutoff,peres2019mixing}, and of random walks on evolving configuration models \cite{AveGulVDHdH:DirConf2016,AveGulVDHdH:DirConf2018}; see also  \cite{sauerwald2019random} for recent results on some related general problems. 
These works are all concerned with the case of undirected graphs, and as far as we know there has been essentially no analysis of mixing times for dynamically evolving directed graphs up to now. Even before discussing the mixing properties, a key problem in the directed setting is the identification of the stationary distribution. In this paper we resolve these difficulties and obtain a precise description of the mixing times for a class of digraphs undergoing a particularly simple evolution, namely for digraphs with given degree sequences that are fully regenerated at independent geometrically distributed random time intervals.

We shall consider two families of directed graphs. Both are obtained via the so-called configuration model, with the difference that in the first case we fix both in and out degrees, while in the second case we only fix the out degrees. The models are sparse in that the degrees are bounded. 
%We start by introducing the two  models of sparse digraphs to be considered. 

 \subsection{Directed configuration model} \label{sec:mod1}
 Let $V$ be a set of $n$ vertices. For simplicity we often write $V=[n]$, with $[n]=\{1,\dots,n\}$.
 For each $n$, we have two finite sequences $\bd^+=(d_x^+)_{x\in[n]}$ and $\bd^-=(d_x^-)_{x\in[n]}$ of non negative integers such that 
 \begin{equation}\label{eq:egs}
 m=\sum_{x\in V}d_x^+=\sum_{x\in V}d_x^-.
 \end{equation} 
 In the {\em directed configuration model} DCM($\bd^\pm$), a random graph $G$ is obtained as follows: 1) equip each node $x$ with a set $E_x^+$ of $d_x^+$ tails and a set $E_x^-$ of $d_x^-$ heads; 2) pick uniformly at random one of the $m!$ bijections from the set of all tails $\cup_x E_x^+$ into the set of all heads $\cup_x E_x^-$, call it $\sigma$; 3) for all $x,y\in V$, add a directed edge $(x,y)$ every time a tail from $x$ is mapped into a head from $y$ through $\sigma$. We call $\cC$ the set of all bijections $\si$, so that $|\cC|=m!$. The resulting graph $G=G(\si)$ may have self-loops and multiple edges, however it is classical that by conditioning on the event that there are no multiple edges and no self-loops one obtains a uniformly random simple digraph with in degree sequence $\bd^-$ and out degree sequence $\bd^+$.  
 
Structural properties of random graphs obtained in this way have been studied in  ~\cite{CooFri:SizeLargSCRandDigr2002}. Here we shall consider the sparse case corresponding to bounded degree sequences, and in order to avoid non irreducibility issues, we shall assume that all degrees are at least $2$. Thus, throughout this work it will always be assumed that 
 \begin{equation}\label{degs}
\min_{x\in[n]}d_x^\pm\ge 2,\qquad\max_{x\in[n]}d_x^\pm=O(1). % \mathbf{1}\lambda
\end{equation}
We often use the notation $\D=\max_{x\in[n]}\max\{d_x^-, d_x^+\}$. Under the first assumption  it is known that DCM($\bd^\pm$) is strongly connected with high probability, while under the second assumption, it is known that DCM($\bd^\pm$) has a uniformly (in $n$) positive probability of having no self-loops nor multiple edges. In particular, any property that holds with high probability for DCM($\bd^\pm$) will also hold with high probability for a uniformly chosen simple graph subject to the constraint that in and out degrees be given by $\bd^-$ and $\bd^+$ respectively. 
Here and throughout the rest of the paper we say that a property holds {\em with high probability} (w.h.p.\ for short) if the probability of the corresponding event converges to $1$ as $n\to\infty$. 

\subsection{Out configuration model} \label{sec:mod2}
To define the second model, for each $n$ let $\bd^+=(d_x^+)_{x\in[n]}$ be a finite sequence of non negative integers. In the {\em out-configuration model} OCM($\bd^+$) a random graph $G$ is obtained as follows: 1) equip each node $x$ with $d_x^+$ tails; 2)  pick, for every $x$ independently, a uniformly random injective map $\sigma_x$  
from  the set of tails at $x$ to the set of all vertices $V$; 3) for all $x,y\in V$, add a directed edge $(x,y)$ if a tail from $x$ is mapped into $y$ through $\sigma_x$.
Equivalently, $G$ is the graph whose adjacency matrix is uniformly random in the set of all $n\times n$ matrices with entries $0$ or $1$ such that every row $x$ sums to $d_x^+$. Notice that $G$ may have self-loops.
%, but there are no multiple edges in this construction. This is due to the requirement that the maps $\sigma_x$ be injective. The latter choice is only a matter of convenience, and everything we say below is actually seen to hold as well for the model obtained by dropping that requirement.   
We write $\sigma=(\sigma_x)_{x\in [n]}$, and let $\cC$ denote the set of all distinct such maps, so that $|\cC|=\prod_{x=1}^n\frac{n!}{(n-d_x^+)!}$. As above we make the assumptions  \begin{equation}\label{degs+}
\min_{x\in[n]}d_x^+\ge 2,\qquad\max_{x\in[n]} d_x^+=O(1), % \mathbf{1}\lambda
\end{equation}
and use the notation $\D=\max_{x\in[n]} d_x^+$. 

\subsection{Mixing of static digraphs}
In what follows $\si\in\cC$ denotes a given realisation of either the directed configuration model  DCM($\bd^\pm$) or the out-configuration model OCM($\bd^+$), and we write $G(\si)$ for the corresponding realisation of the digraph. We will treat both models on an equal footing as much as possible,   and when we need to distinguish between them
we often refer to these as {\em model 1} and {\em model 2} respectively.   
When the underlying graph is static and is given by one of these two models, the mixing time of the random walk has been studied in \cite{BCS1,BCS2}. For a fixed configuration $\si\in\cC$, we consider the transition matrix
\begin{equation}\label{eq:pisi}
P_\si(x,y)=\frac{\#(\si; x\to y)}{d_x^+},\qquad x,y\in[n],
\end{equation}
where $\#(\si; x\to y)$ denotes the number of directed edges from $x$ to $y$ in $G(\si)$. The random walk on $G(\si)$ is thus the Markov chain $(X_0,X_1,\dots)$ with state space $[n]$ and with transition probabilities $P_\si(x,y)$. We use the notation   $\mathbf{P}^\si_x(\cdot)$ for the law of the trajectory $(X_0,X_1,\dots)$ when $X_0=x$, so that in particular, for any $x,y\in[n]$, and $t\in\bbN$:
\begin{equation}\label{eq:pisi2}
\mathbf{P}^\si_x(X_t=y)= P_\si^t(x,y).
\end{equation}
We remark that in each of the two models above, the random walk on the digraph $G=G(\si)$ has with high probability a unique stationary distribution $\pi_\si$. In model 1 this follows from the fact that $G(\si)$ is w.h.p.\ strongly connected \cite{CooFri:SizeLargSCRandDigr2002}. In model 2 on the other hand $G(\si)$ may have vertices with in-degree zero (or, more generally, with a bounded in-neighborhood) and one cannot conclude that $G(\si)$ is strongly connected.  However, it is still the case that w.h.p.\ there exists a unique 
stationary distribution; see e.g.\ \cite{AbBP,BCS2} for more details. Let us now recall the main results of \cite{BCS1,BCS2}. 
Convergence to equilibrium will be quantified using the total variation distance. Given two probability measures $\mu,\nu$, the latter is defined by
\begin{equation}\label{def:tv}
\|\mu-\nu\|_\tv = \max_E |\mu(E)-\nu(E)|,
\end{equation}
where the maximum ranges over all possible events in the underlying probability space. Let the {\em entropy} $H$ and the associated {\em entropic time} $\tent$ be defined by 
\begin{equation}\label{def:tent}
H=\sum_{x\in V}\muin(x)\log d^+_x, \;\qquad \tent = \frac{\log n}H,
\end{equation}
where the probability $\muin$ is defined as $\muin(x)=d_x^-/m$ for model 1 and as $\muin(x)=1/n$ for model 2. 
Note that our assumptions on $\bd^\pm$ imply that the deterministic quantities $H,\tent$ satisfy $H=\Theta(1)$ and $\tent = \Theta(\log n)$. 
\begin{theorem}[Uniform cutoff at the entropic time \cite{BCS1,BCS2}]\label{th:BCS}
Let $G(\si)$ be a random graph from either
the directed configuration model  DCM($\bd^\pm$) or the out-configuration model 
OCM($\bd^+$).
For each $\beta>0, \beta\neq 1$ one has:  
\begin{equation}\label{eq:cutoffo}
\max_{x\in [n]}\left| \|P_\si^{t}(x,\cdot)-\pi_\si\|_\tv-  \vartheta(\beta)
\right|\overset{\P}{\longrightarrow}0,\qquad t=\lfloor \beta\tent\rfloor,
\end{equation}
where $\vartheta(\beta)$ is the step function 
$\vartheta(\beta)=\ind(\beta< 1)$. 
\end{theorem}
In \eqref{eq:cutoffo} we use the notation $\overset{\P}{\longrightarrow}$ for convergence in probability as $n\to\infty$ with respect to the random choice of the configuration $\si\in\cC$.
If we define the mixing time $T_\si(\e)$ as 
the first $t\in\bbN$ such that $ \max_{x\in [n]}\|P_\si^{t}(x,\cdot)-\pi_\si\|_\tv\leq \e$, then Theorem \ref{th:BCS} establishes 
that the mixing time of the random walk on the static digraph satisfies with high probability $T_\si(\e)=(1+o(1))\tent$, for any fixed $\e\in(0,1)$.  This is an instance of the so-called cutoff phenomenon \cite{Dcutoff}. We refer to \cite{LS,ben2017cutoff,berestycki2018random} for related results in the context of undirected graphs. 

Another important fact established in \cite{BCS1,BCS2} is that relaxation to equilibrium occurs much earlier than  the mixing time if one starts from a delocalized initial state. See also \cite{SC} for a related result in terms of the spectrum of the matrix $P_\si$. The precise version of this fact that we shall need reads as follows. Call {\em widespread} a sequence of probability measures $\l=\l_n$ on $[n]$ such that for some $\e>0$:
\begin{equation}\label{eq:widesp}
\max_{x\in [n]}\l(x) \leq n^{-\frac12-\e}\,,\quad \text{and}\quad \limsup_{n\to\infty} \sum_{x\in[n]}n\lambda_n(x)^2<\infty.
\end{equation}
For any probability measure $\l$ on $[n]$ we write $\l P^t_\si$ for the probability $\sum_x \l(x)P^t_\si(x,y)$. The following result is proven in \cite[Lemma 5]{CQ:PageRank}. 
 \begin{lemma}\label{le:approx}
Let $G(\si)$ be a random graph from either
the directed configuration model  DCM($\bd^\pm$) or the out-configuration model 
OCM($\bd^+$). If $\l$ is widespread, then for any sequence $t=t(n)\to\infty$, 
\begin{equation}\label{eq:wsapprox}
\left\|\l P_\si^t - \pi_\si\right\|_{\tv}\overset{\P}{\longrightarrow} 0.
		\end{equation}
\end{lemma}
As we shall see in Corollary \ref{co:widespread-pi} below, the probability $\pi_\si$ is itself widespread with high probability. 

\subsection{Mixing of dynamic digraphs}
We now introduce the joint evolution of the digraph and the random walk. Given $\a\in(0,1)$, we consider the Markov chain with state space $\cC\times [n]$ and
with transition matrix 
\begin{equation}\label{eq:transition}
\cP_\alpha((\sigma,x),(\eta,y))=(1-\alpha)P_\sigma(x,y)\ind_{\sigma}(\eta)+\alpha \,\bu(\eta)\ind_{x}(y),		
\end{equation}
where $\ind_a(b)$ stands for $1$ if $a=b$ and $0$ otherwise and $\bu(\eta)=|\cC|^{-1}$ denotes the uniform distribution over the set $\cC$. In words, at each time $t\in \bbN$ independently, we sample a Bernoulli$(\a)$ random variable $J_t$; if $J_t=1$ we pick a uniformly random  $\eta\in\cC$ and move from the current state $(\si,x)$ to the new state $(\eta,x)$, while if $J_t=0$ we move to the new state $(\si,y)$ where $y$ is chosen uniformly at random among the out-neighbours of $x$ in the digraph $G(\si)$.   We write $\{(\xi_t,X_t),\,t\geq 0\}$ for the trajectory of the Markov chain and write $\Pj_{\si,x}(\cdot)$ for its law when started at $\xi_0=\si$ and $X_0=x$. 
It is not hard to check that this is %$(\xi_t,X_t)_{t\ge0}$ is 
an irreducible and aperiodic Markov chain and therefore it admits a unique stationary distribution $\pi^{\rm J}$. 
A consequence of our results, see Remark \ref{rem:stat} below, is that $\pi^{\rm J}$ is well approximated in total variation distance by the probability measure $\nu$ on $\cC\times [n]$ defined by
\begin{equation}\label{eq:def-nu}
\nu(\sigma,x)=\bu(\sigma)\pi_\sigma(x).
\end{equation}
We know that $\pi_\si$ is uniquely defined for all $\si$ in a set $\Omega_n\subset \cC$ with $\bu(\Omega_n)\to 1$ as $n\to\infty$. To extend $\nu$ to all $  \cC\times [n]$ we may define e.g.\ $\pi_\si=\muin$ for $\si\in\cC\setminus \Omega_n$.  
%We will show that $\nu$ is a good $\tv$-proxy for $\pi^J$. 
%\begin{comment}\label{th:stationary-joint}
%Call 
%$$\cE_n(\delta,\Delta)=\{\bd^{\pm}\:|\: \max_{x\in[n]}d^{\pm}_x\le\Delta, \min_{x\in[n]} d^{\pm}_x\ge\delta \}$$
%Then, for every $\alpha\in(0,1)$  the chain $(G_t,X_t)_{t\ge 0}$ is ergodic, in the sense that there exists some $N_{\delta,\Delta}\in\N$ such that fixed $n>N_{\delta,\Delta}$, for every 
%$\bd^{\pm}\in \cE_n(\delta,\Delta)$ there exists some $\pi^J\equiv\pi^J_n(\bd^{\pm})$ for which
%$$\lim_{t\to\infty}\sup_{\sigma\in \cC}\sup_{x\in[n]}\|\cP_\a^t((\sigma,x),(\cdot,\cdot))-\pi^J \|_{\tv}=0.$$
%Moreover,
%$$\lim_{n\to\infty}\|\nu-\pi^J\|_{\tv}=0.$$
%\end{comment}
%The latter result is achieved by mean of sharp bounds on the sequence parametrized by $t\ge 1$
We define 
\begin{equation}\label{eq:def-daj}
\cD_{\sigma,x}^{{\rm J},\alpha}(t)=\| \Pj_{\sigma,x}(\xi_t=\cdot,X_t=\cdot)-\nu\|_{\tv}.
\end{equation}
For each $t\in\bbN$, the quantity $\cD_{\sigma,x}^{{\rm J},\alpha}(t)$ is regarded as a random variable with respect to the uniform choice of the configuration $\si\in\cC$.  
Moreover, we extend  $\cD_{\sigma,x}^{{\rm J},\alpha}(\cdot)$ to all positive reals by taking the integer part of the argument.  

\begin{theorem}\label{th:trichotomy-joint}
Fix a sequence $\alpha=\a_n$  such that $\a_n\to 0$ as $n\to\infty$. Then, for all $\beta>0$
\begin{equation}\label{eq:up-daj}
\limsup_{n\to\infty}\max_{\si\in\cC,x\in[n]}\cD_{\sigma,x}^{{\rm J},\alpha}(\beta\a^{-1})\leq (1+\beta)e^{-\beta}. 
\end{equation}
Next, assume  that 
\begin{equation}\label{eq:def-ga}
\gamma=\lim_{n\to\infty}\alpha\, T_\ent\in [0,\infty].
\end{equation}
Then, according to the value of $\gamma$ there are three scenarios:
%
%Then for every constant $\beta>0$, independently on $\gamma$ it holds the upper bound
%%	$$\liminf_{n\to\infty}\inf_{\sigma,x}\cD_{\sigma,x}^{{\rm J},\alpha}(\beta\a^{-1})\ge e^{-\beta},$$
%%	and the upper bound
%$$\limsup_{n\to\infty}\max_{\sigma\in\cC}\max_{x\in[n]}\cD_{\sigma,x}^{{\rm J},\alpha}(\beta\a^{-1})\le(1+\beta) e^{-\beta}.$$
%Now let $\P$ represent the law with respect to the sampling of the starting environment $G_0\sim \text{Unif}(\cC)$.
\begin{enumerate}
	\item 
	If $\gamma=0$ then for all $\beta>0$:
	\begin{equation}\label{scena1}
	\max_{x\in[n]}\left| \cD_{\sigma,x}^{{\rm J},\alpha}(\beta\a^{-1})-e^{-\beta}\right|\overset{\P}{\longrightarrow}0.
		\end{equation}

	\item If $\gamma=\infty$ then for all $\beta>0$:
	\begin{equation}\label{scena2}
	\max_{x\in[n]}\left| \cD_{\sigma,x}^{{\rm J},\alpha}(\beta\a^{-1})-(1+\beta)e^{-\beta}\right|\overset{\P}{\longrightarrow}0.
		\end{equation}
		%$$\max_{x\in[n]}\left| \cD_{\sigma,x}^{{\rm J},\alpha}(\beta\a^{-1})-(1+\beta)e^{-\beta}\right|\overset{\P}{\to}0.$$
	%\item If $\gamma=\infty$:
	%$$\limsup_{n\to\infty}\sup_{\sigma\in\cG,x}\cD_{\sigma,x}^{{\rm J},\alpha}(t)\le (1+\beta)e^{-\beta} .$$
	\item If $\gamma\in(0,\infty)$ then for all $\beta>0$, $\beta\not=\gamma$:
		\begin{equation}\label{scena3}
		\max_{x\in[n]}\left| \cD_{\sigma,x}^{{\rm J},\alpha}(\beta\a^{-1})-\psi_\gamma(\beta)\right|\overset{\P}{\longrightarrow}0.
		\end{equation}
%	$$\limsup_{n\to\infty}\sup_{\sigma\in\cG,x}\cD_{\sigma,x}^{{\rm J},\alpha}(\beta\a^{-1})\le\vartheta_\gamma(\beta),$$
	where
	$$\psi_\gamma(\beta)=\begin{cases}
	(1+\beta)e^{-\beta}&\text{if }\beta<\gamma\\
	e^{-\beta}&\text{if }\beta>\gamma
	\end{cases}.$$
\end{enumerate}
\end{theorem}
\begin{figure}
	\centering
	\includegraphics[width=4.2cm]{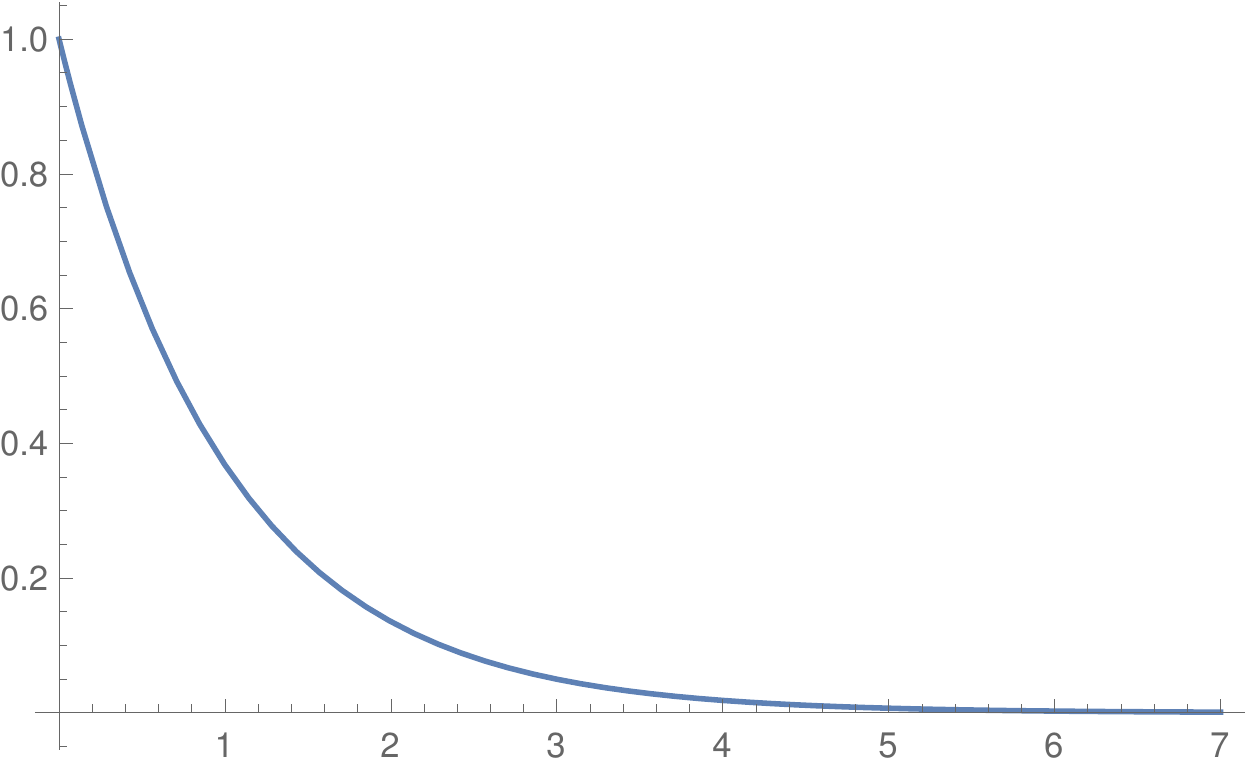}\qquad 
	\includegraphics[width=4.2cm]{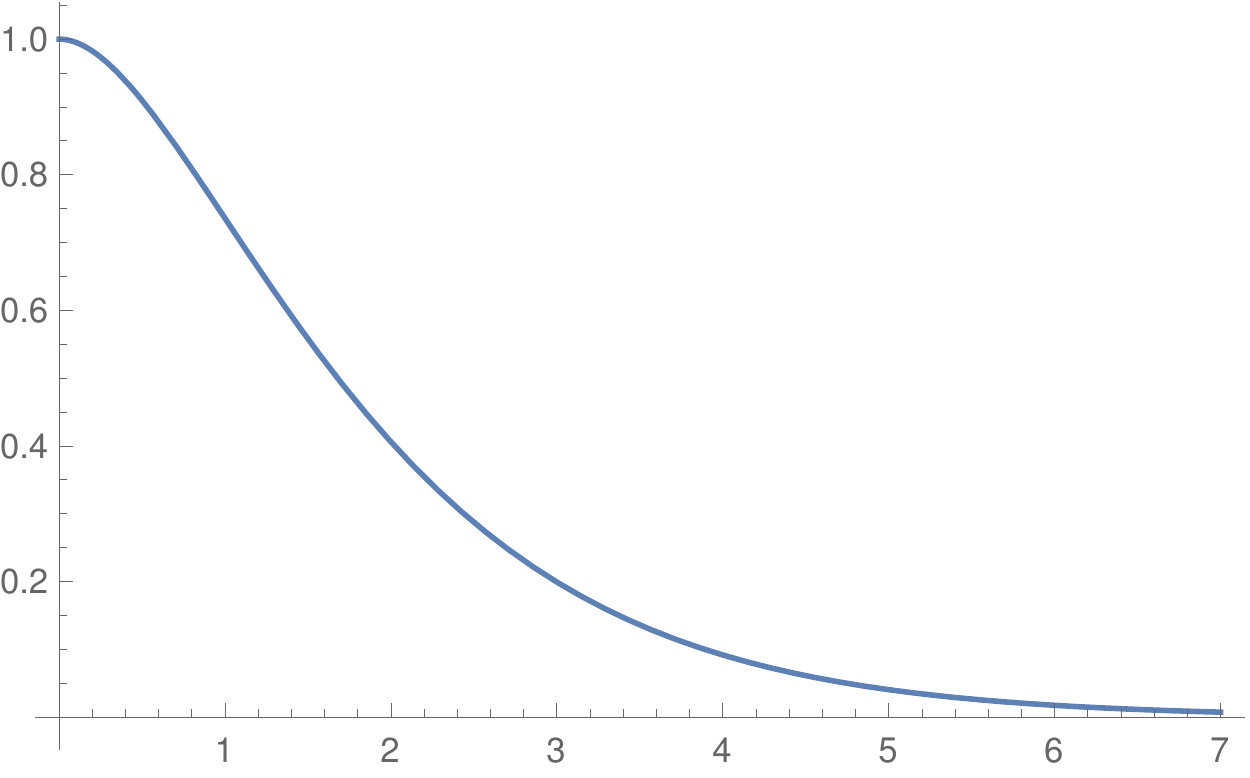}
	\qquad 
	\includegraphics[width=4.2cm]{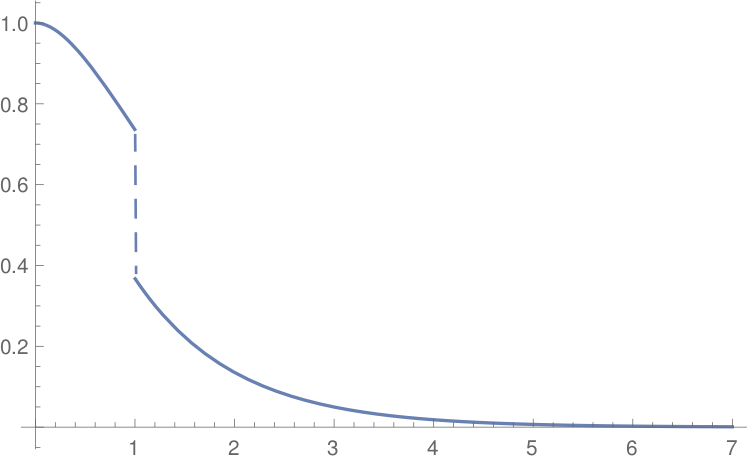}
	\caption{The asymptotic behavior on the scale $\alpha^{-1}$ of the quantity $\cD_{\sigma,x}^{{\rm J},\alpha}(t)$ for a typical starting environment $\sigma$ and arbitrary $x\in[n]$ in the case $\gamma=0$ (left), $\gamma=\infty$ (center) and $\gamma\in\(0,\infty \)$ (right). The transition point in this last scenario is $t=\gamma\alpha^{-1}\sim\tent$, and we have taken $\gamma=1$.}
\label{fig:fig1}
\end{figure}
The trichotomy displayed in Theorem \ref{th:trichotomy-joint} can be interpreted as follows; see 
also Figure \ref{fig:fig1}. On the time scale $\alpha^{-1}$ the regeneration times, that is the $t\in\bbN$ such that $J_t=1$, converge to a Poisson process of intensity $1$. Then $e^{-\beta}$ and $(1+\beta)e^{-\beta}$ represent the probability of having no regeneration and at most one regeneration up to time $\beta\a^{-1}$ respectively. 
Thus Theorem \ref{th:trichotomy-joint} essentially says that when the walk is far from being mixed within the current digraph then  two regenerations are necessary and sufficient for a complete loss of memory of the initial state, whereas if the walk has already mixed within the current digraph then all it is required to reach stationarity is one regeneration. 

\begin{remark}\label{rem:stat} 
From Theorem \ref{th:trichotomy-joint} it follows that 
\begin{equation}\label{eq:stati}
\lim_{n\to\infty}\|\nu-\pi^{\rm J}\|_\tv = 0,
\end{equation}
which in turn implies that all statements in Theorem \ref{th:trichotomy-joint} hold
with $\nu$ replaced by $\pi^{\rm J}$. 
Indeed, to prove \eqref{eq:stati} observe that 
the invariance $\pi^{\rm J}\cP_\a =\pi^{\rm J}$ implies that for any $t\in\bbN$
%taking $t=\beta\a^{-1}$ and letting $\beta\to\infty$ at fixed $n$, \eqref{eq:up-daj} implies
\begin{align*}\label{eq:statio}
\|\nu-\pi^{\rm J}\|_\tv& \leq %\lim_{\b\to\infty}\Big\|\displaystyle{\nu-\frac1n\sum_{\si\in\cC,x\in[n]}\bu(\si)\Pj_{\sigma,x}(\xi_{\beta\a^{-1}}=\cdot,X_{\beta\a^{-1}}=\cdot)}\Big\|_\tv\\& 
\max_{\si\in\cC,x\in[n]}\cD_{\sigma,x}^{{\rm J},\alpha}(t).
\end{align*}
Taking $t=\beta\a^{-1}$, \eqref{eq:up-daj} implies that $\limsup_{n\to\infty}\|\nu-\pi^{\rm J}\|_\tv\leq (1+\beta)e^{-\beta}$, and letting $\beta\to\infty$ we obtain \eqref{eq:stati}.
\end{remark}

The proof of Theorem \ref{th:trichotomy-joint} will be crucially based on Theorem \ref{th:BCS} and Lemma \ref{le:approx}. The dynamic setting however requires an important extension of these results that can be formulated as follows. 
For any $(\si,\eta)\in \cC\times\cC$ and integers $0\leq s\leq t$, define
\begin{equation}\label{eq:qst}
Q^{s,t}_{\si,\eta}(x,y)=\sum_{z\in [n]}P_\sigma^s(x,z)P_\eta^{t-s}(z,y).
\end{equation}
%
%
%In order to prove the trichotomy in Theorem \ref{th:trichotomy-joint} we need a generalization of the cutoff proved in \cite{BCS1,BCS2}. The following result is achieved by an adaptation of the techniques therein.
Notice that the following theorem reduces to Theorem \ref{th:BCS} if $s=0$.
	\begin{theorem}[Cutoff on double digraphs]\label{th:2cutoff} 
			Fix $\beta>0$, take $t=\beta T_\ent$. %and let $s=s(n)$ be any sequence such that $0\leq s\le t$. 
			%Fixed any couple $(\si,\eta)\in \cC\times\cC$ we consider the probability distribution on $[n]$
	%	$$[Q_s^t(x,y)](\si,\eta)\equiv Q_s^t(x,y):=\sum_{z\in [n]}P_\sigma^s(x,z)P_\eta^{t-s}(z,y).$$
Let $\sigma$ and $\eta$ be two independent uniformly random configurations in $\cC$, and let $\P$ denote the associated probability. Then for fixed $\beta>0$:
\begin{enumerate}
	\item If $\beta<1$, 
	\begin{equation}\label{lowbobo}
\min_{s\in[0,t]}\min_{x\in[n]}\| Q^{s,t}_{\si,\eta}(x,\cdot)
%	Q_s^t(x,\cdot)
-\pi_\eta\|_{\tv}\overset{\P}{\longrightarrow}1.
\end{equation}
	\item If $\beta>1$, then for every $0<\varepsilon\leq \min\{1,(\beta-1)/2\}$ %and $a\in(0,1)$ are fixed constants, %then %t-s\to\infty$ 
	%as $n\to\infty$,
		\begin{equation}\label{upbobo}
		\max_{s\in I(\varepsilon,\beta)}\max_{x\in[n]}\| 
		%Q_s^t(x,\cdot)
		Q^{s,t}_{\si,\eta}(x,\cdot)-\pi_\eta\|_{\tv}\overset{\P}{\longrightarrow}0,
		\end{equation}
\end{enumerate}
where $I(\varepsilon,\beta)=[0,(1-\varepsilon)\tent]\cup [(1+\varepsilon)\tent, (\beta-\varepsilon)\tent]$. 
	\end{theorem}
	
Theorem \ref{th:2cutoff} is the technical core of the paper. In order to simplify the exposition we have chosen to give the proof of the main results in Section \ref{se:joint} and Section \ref{trichRW} by assuming the validity of Theorem \ref{th:2cutoff}, which will be later proved in  Section \ref{sec:2cutoff}. 	We refer to Remark \ref{re:sologn} for more details on the requirement that $s\in I(\varepsilon,\beta)$ in Theorem \ref{th:2cutoff}.

Another key ingredient for the proof of Theorem \ref{th:trichotomy-joint} is the control of the annealed walk. By this we mean the law
\begin{equation}\label{eq:ann}
\Pan_x(\cdot)=\sum_{\eta\in\cC}\bu(\eta) \mathbf{P}^\eta_x(\cdot),
\end{equation}	
where $\mathbf{P}^\eta_x$ is defined before \eqref{eq:pisi2}.
\begin{lemma}\label{le:general}
The annealed law satisfies 
\begin{equation}\label{eq:anntv}
\lim_{n\to\infty}\sup_{x\in[n],t\ge 1}\|\Pan_x(X_t=\cdot) -\muin\|_{\tv}=0.
\end{equation}
\end{lemma}
In words, the marginal at fixed time $t\geq 1$ of the annealed law is uniformly well approximated by the in-degree distribution $\muin$. This is a consequence of the directed nature of the network, which makes the walk essentially non-backtracking. 

Lemma \ref{le:general} will be proved in Section \ref{se:joint}. Once Lemma \ref{le:general} and Theorem \ref{th:2cutoff} are available, we shall obtain Theorem \ref{th:trichotomy-joint} by a decomposition of the law at time $t$ according to the location of the regeneration times; see Section \ref{sec:trich}. 

Finally, our last main result concerns the marginal distribution of the position of the walk, namely the non-Markovian process obtained by projecting the chain $(\xi_t,X_t)_{t\geq 0}$ on the second coordinate. According to Theorem \ref{th:trichotomy-joint} and Lemma \ref{le:general} the law of $X_t$, for $t$ and $n$ suitably large,  should be well approximated by $\muin$. 
The next result quantifies this statement by exhibiting once again a trichotomy.   
%\begin{equation}\label{eq:def-marg}
%\sum_{\sigma\in \cC}\nu(\sigma,x).
%\end{equation}
%As we will see below, it is not hard to check that in both models the marginal in \eqref{eq:def-marg} is asymptotically near to the measure in \eqref{eq:def-muin}. In fact, we will show that
%$$\left\|\sum_{\si\in\cC}\nu(\si,\cdot)-\muin \right\|_{\tv}=o(1).$$
Define 	
\begin{equation}\label{eq:dista}
\cD_{\sigma,x}^{\a}(t):=\|\Pj_{\sigma,x}(X_t=\cdot)-\muin \|_{\tv}\,,\qquad q:= \bbE\|\pi_\si-\muin \|_{\tv}.
\end{equation}
%	The following theorem gives sharp asymptotic bounds on the function $\cD_{\sigma,x}^{\a}(t)$, showing again that a trichotomy takes place on the scale $t=\Theta(\alpha)$, in dependence with the asymptotic behavior of $\alpha T_\ent$. 
We remark that if the sequences $\bd^\pm$ are {\em Eulerian}, that is $d^+_x=d^-_x$ for all $x\in[n]$, then $\pi_\si = \muin$ is stationary for all $\si\in\cC$. Thus in this case $q=0$.
On the other hand, results from \cite{Liu4,Liu1} imply that if the sequence is not eulerian then $q$ is bounded away from 0 and 1; see \cite[Theorem 4]{BCS1} and \cite[Remark 19]{CQ:PageRank} for more details.  

\begin{theorem}\label{th:trichotomy-rw}
Fix a sequence $\alpha=\a_n$  such that $\a_n\to 0$ as $n\to\infty$. Then, for all $\beta>0$
\begin{equation}\label{eq:up-daj2}
\limsup_{n\to\infty}\max_{\si\in\cC,x\in[n]}\cD_{\sigma,x}^{\alpha}(\beta\a^{-1})\leq e^{-\beta}. 
\end{equation}
Next, assume  that 
%
%	$$\gamma:=\lim_{n\to\infty}\alpha T_\ent.$$
%	Then for every constant $\beta>0$, independently on $\gamma$ it holds the upper bound
%	$$\limsup_{n\to\infty}\max_{\si\in\cC}\max_{x\in[n]}\cD_{\si,x}^{\a}(\beta\a^{-1})\le e^{-\beta}.$$
%	Moreover, let $\P$ be the law with respect to the sampling of the starting environment $G_\sim\text{Unif}(\cC)0$. Then there exists a constant $C\equiv C_n(\bd)\in(0,1)$ such that for every constant $\beta>0$
\begin{equation}\label{eq:def-ga2}
\gamma=\lim_{n\to\infty}\alpha\, T_\ent\in [0,\infty].
\end{equation}
Then, according to the value of $\gamma$ there are three scenarios:
	\begin{enumerate}
		\item If $\gamma=0$ then for all $\beta>0$:
		\begin{equation}
		\max_{x\in[n]}\left|\cD_{\sigma,x}^{\a}(\beta\alpha^{-1})- q\, e^{-\beta}\right|\overset{\P}{\longrightarrow}0,
		\end{equation}
		and, if $\beta\not=1$ then
			\begin{equation}
			\max_{x\in[n]}\left|\cD_{\sigma,x}^{\a}(\beta T_\ent)-\varphi(\beta )\right|\overset{\P}{\longrightarrow}0,
			\end{equation}
			where
			\begin{equation}\label{eq:varphi}
			\varphi(\beta):=\begin{cases}
			1&\text{ if }\beta<1\\
			q&\text{ if }\beta>1.
			\end{cases}
			\end{equation}
		\item If $\gamma=\infty$,  then for all $\beta>0$:
		\begin{equation}
		\max_{x\in[n]}\left|\cD_{\sigma,x}^{\a}(\beta \a^{-1})-e^{-\beta}\right|\overset{\P}{\longrightarrow}0.
		\end{equation}
		\item If $\gamma\in(0,\infty)$ then for all $\beta>0$, $\beta\not=\gamma$:
			\begin{equation}
			\max_{x\in[n]}\left|\cD_{\sigma,x}^{\a}(\beta\a^{-1})-\varphi(\beta/\gamma)e^{-\beta}
			\right|\overset{\P}{\longrightarrow}0.
			\end{equation}
%		\begin{equation}
%		\varkappa_\gamma(\beta ):=\begin{cases}
%		e^{-\beta}&\text{ if }\beta<\gamma\\
%		q\,e^{-\beta}&\text{ if }\beta>\gamma
%		\end{cases}
%		\end{equation}
	\end{enumerate}
\end{theorem}
\begin{figure}[h]
	\centering
		\includegraphics[width=4.2cm]{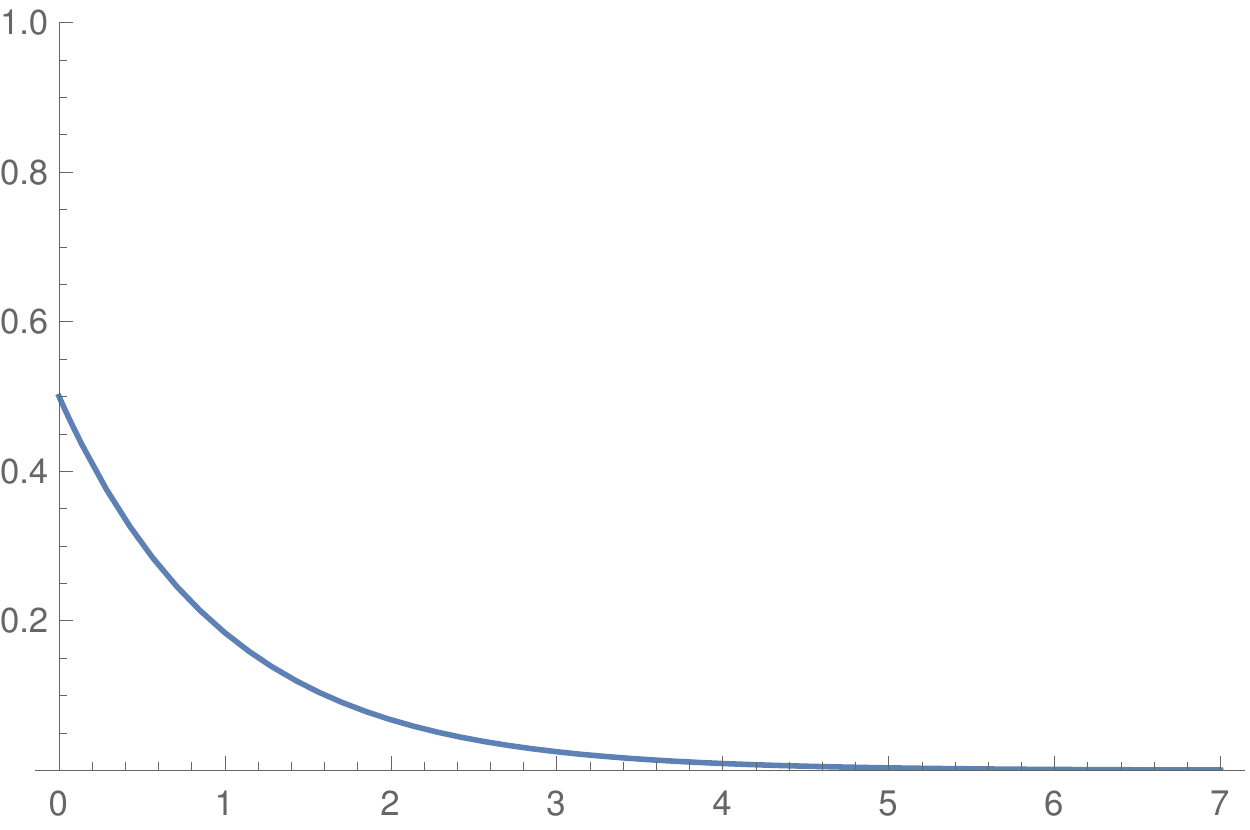}\qquad
		\includegraphics[width=4.2cm]{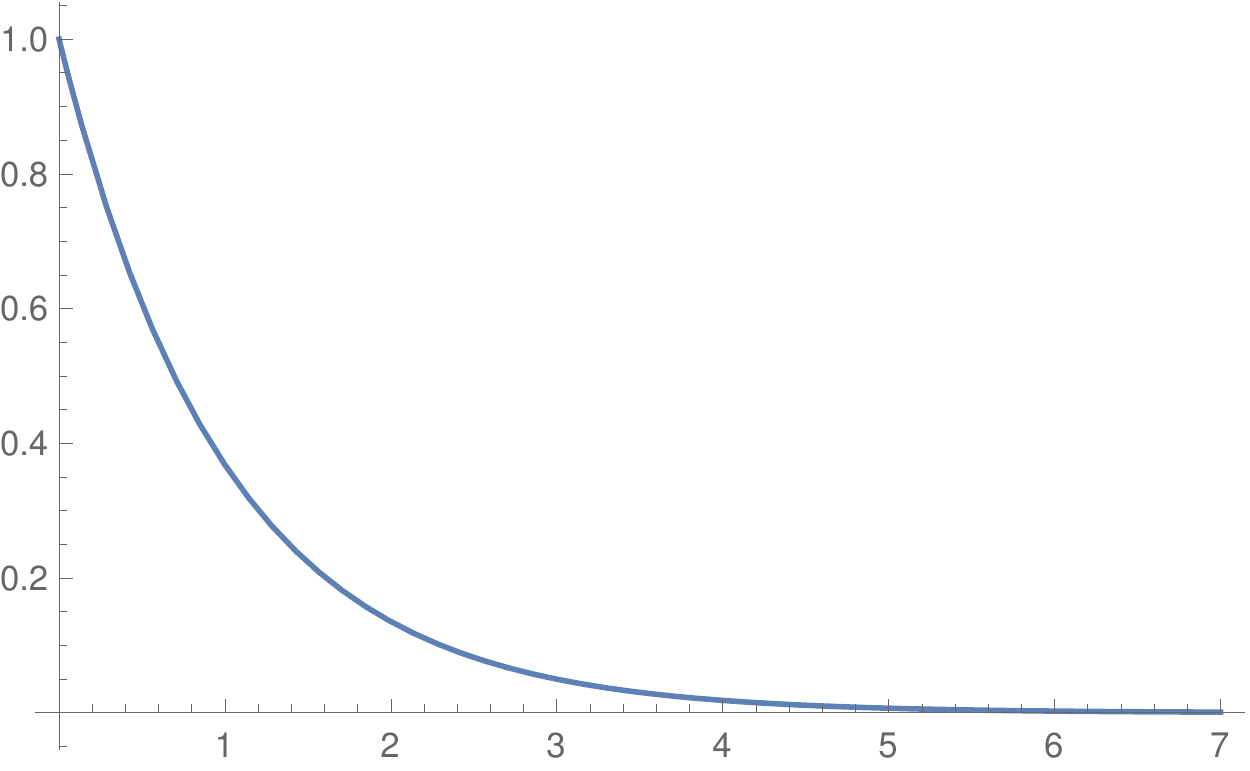}\qquad
	\includegraphics[width=4.2cm]{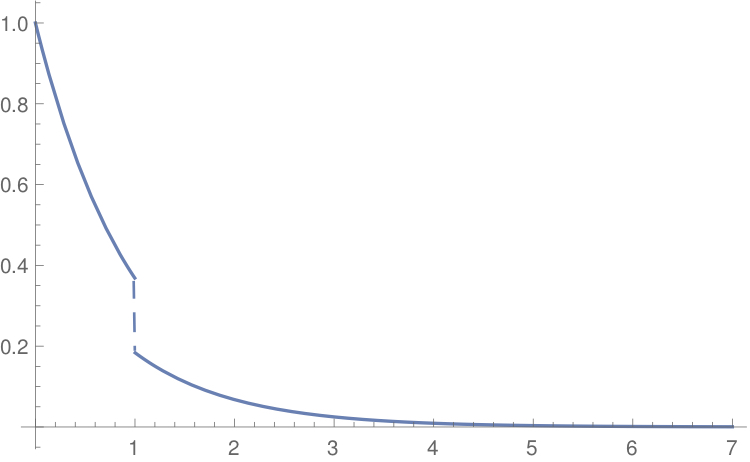}\caption{The asymptotic behavior on the scale $\alpha^{-1}$ of the quantity $\cD_{\sigma,x}^{\alpha}(t)$ for a typical starting environment $\sigma$ and arbitrary $x\in[n]$ in the case $\gamma=0$ (left), $\gamma=\infty$ (center) and $\gamma\in\(0,\infty \)$ (right). The transition point in the latter case is $t=\gamma\alpha^{-1}\sim \tent$. In this picture we take $\gamma=1$ and $q=1/2$.}
\label{fig:fig2}
\end{figure}
The above results can be roughly interpreted as follows. If we follow only the position of the particle then after the first regeneration time the walk has the annealed law, and by Lemma \ref{le:general} this is approximately $\muin$. Thus, a complete loss of memory of the initial state with relaxation to the limiting state $\muin$ occurs essentially at the time of the first regeneration of the digraph. On the other hand, if no regeneration occurs, then a partial loss of memory occurs at time $\tent$ because of the static mixing cutoff phenomenon, and this is quantified by the drop by a factor $q$ in total variation. The competition between these two effects explains the above triad; see Figure \ref{fig:fig2}. 

We conclude this introduction with some remarks on related work and a comment on possible extensions. 
The first instance of a trichotomy in the relaxation to equilibrium for random walk on dynamic graphs was revealed in \cite{AveGulVDHdH:DirConf2016,AveGulVDHdH:DirConf2018}.
The authors studied non-backtracking walks on undirected graphs undergoing partial regenerations and obtained results that are qualitatively similar to our Theorem \ref{th:trichotomy-rw}, with the difference that the quantity analogous to $q$ is zero in their case. 
While their model allows for more general regeneration mechanisms than the one considered here, a strongly simplifying feature of their setting with respect to ours is that the stationary distribution is not altered when the underlying graph is updated, and it has a simple explicit form. Inspired by these works, analogous trichotomy results were obtained for the PageRank surfer on static digraphs \cite{CQ:PageRank}. As we observed in \cite{CQ:PageRank}, teleportation in the PageRank process plays a role similar to the dynamic regeneration; see also \cite{vial2019restart,wang2019regeneration} for related developments. 

An interesting extension of the results presented here would be to consider partial regenerations of the underlying digraph instead of full regenerations. For instance, a  natural dynamic model for the DCM($\bd^\pm$) can be obtained by updating the permutation $\si$ realising the directed configuration using random transpositions only. Mixing time and the cutoff phenomenon are well understood for random transpositions \cite{diaconis1981generating}, and it seems natural to seek for analogues of the results %of the same kind of those 
obtained here in that finer setting.

\section{Trichotomy for the joint process}\label{se:joint}
We start with the proof of Lemma \ref{le:general}, and then prove Theorem \ref{th:trichotomy-joint} assuming the validity of Theorem \ref{th:2cutoff}. The proof of the latter is given in Section \ref{sec:2cutoff} below. 
\subsection{Proof of Lemma \ref{le:general}}
	We divide the proof in two cases: %, which we will refer to as \emph{small $t$}, namely 
	$t\le 2 T_\ent$, 
	and 
	%\emph{large $t$}, i.e., 
	$t>2 \tent$. If $t\leq 2\tent$, in particular one has $t=O(\log n)$,  and 
	%is \emph{small}, 
we know from \cite[Lemma 3.10]{CQ:cover} that %for model 1 that %uniformly in $x\not=y$
	\begin{equation}\label{eq:est1}
	\Pan_x(X_t=y)=\muin(y)(1+o(1)), \qquad \Pan_x(X_t=x)=O\(n^{-1}\log n\), 	
	\end{equation}
	for $t=O(\log n)$,
uniformly in $x,y\in[n]$. Moreover, the estimates \eqref{eq:est1} are uniform in $t\leq C\log n$, for any fixed constant $C>0$. The proof of \eqref{eq:est1} is carried out in detail in  \cite[Lemma 3.10]{CQ:cover} for model 1 only, but the very same arguments imply the validity of the statements for model 2 as well. % is given 
	The estimates in \eqref{eq:est1} %and \eqref{eq:est2} 
	are enough to conclude that uniformly in $x\in[n]$,
	\begin{align}\label{eq:small-t}
	\|\Pan_x(X_t=\cdot)-\muin\|_\tv=\frac12\sum_{y\in[n]}| \Pan_x(X_t=y)-\muin(y)|=o(1)\,,\qquad t=O(\log n).
	\end{align}
	We now turn to the case $t>2 \tent$. %attention on the \emph{large $t$} scenario. 
	By the triangle inequality we have
	\begin{align}\label{eq:dis-tri}
	\|\Pan_x(X_t=\cdot)-\muin\|_\tv\leq 
	%\sum_{y\in[n]}\left|\Pan(X_t=y)-\muin(y) \right|\le 
	\bbE\|P_\si^t(x,\cdot)-\pi_\si\|_\tv + \|\bbE\pi_\si-\muin\|_\tv.
	\end{align}
	Here the expectation is with respect to the uniform distribution over $\si$ and we have used the identity $\Pan_x(X_t=\cdot)=\bbE P_\si^t(x,\cdot)$.
%	can bound
%	\begin{align}\label{eq:(A)}
%	\sum_{y\in[n]}\left| \E[P^t(x,y)-\pi(y)]\right|\le&\sum_{y\in[n]}\E\left| P^t(x,y)-\pi(y)\right|\\
%	\nonumber=&\sum_{\eta \in \mathcal{G}}\bu(\eta)\sum_{y\in[n]}\left|P^t_{\eta}(x,y)-\pi_{\eta}(y)\right| + o(1) \\
%	\nonumber=&2\sum_{\eta\in\cG}\bu(\eta)\| P^t_\eta(x,\cdot)-\pi_\eta\|_{\tv}+o(1)=o(1)
%	\end{align}
	%where we used 
	From Theorem \ref{th:BCS}, and the fact that the total variation distance is bounded, we obtain 
	\begin{align}\label{eq:(A)}
	\bbE\|P_\si^t(x,\cdot)-\pi_\si\|_\tv=o(1)\,,
	\end{align} uniformly in $x\in[n]$, and $t>(1+\e)\tent$, for any fixed $\e>0$. 
Thus  the first term on the right hand side of \eqref{eq:dis-tri} is $o(1)$.  
Concerning the second term on the right hand side of \eqref{eq:dis-tri}, using again the triangle inequality and setting $s=\lfloor2\tent\rfloor$, we write
	\begin{align}
	\|\bbE\pi_\si-\muin\|_\tv\leq \bbE\|P_\si^s(x,\cdot)-\pi_\si\|_\tv + \|\Pan_x(X_s=\cdot)-\muin\|_\tv =o(1),
	\end{align}
	where the first term is estimated using  \eqref{eq:(A)} and the second term is estimated using \eqref{eq:small-t}. 
This ends the proof of Lemma \ref{le:general}.

\subsection{Proof of Theorem \ref{th:trichotomy-joint}}\label{sec:trich}
For every $(\eta,y)\in \cC\times[n]$, define
$$\mu^{\si,x}_t(\eta,y)=\Pj_{\si,x}(\xi_t=\eta,X_t=y).$$
Recall that $J_s$, $s\in\bbN$ are i.i.d.\ Bernoulli($\a$) random variables indicating the occurrence of the regeneration event.  
For each $t\geq 1$, consider the random variable $\t=\t(t)$ defined by
%We consider the random times
\begin{equation}\label{eq:rts}
%\tau=\inf\{s\in\N\:|\: J_s=1 \}\,,\qquad 
\tau=\ind\left(\exists s\in\{1,\dots,t\}: \,J_s=1\right)\sup\{s\le t\:|\: J_s=1 \}.
\end{equation}
%$$\tau^\circ=\inf\{s\in\N\:|\: J_s=1 \}.$$
%If $\tau^\circ\le t$ we set
%$$%\tau^{last}(t)\equiv
%\tau^{last}:=\sup\{s\le t\:|\: J_s=1 \}.$$
%Let us partition the space of joint paths of length $t$ in the sets
%$$\{\tau^\circ>t \},\qquad\{\tau^{last}=s\},\:s\in[t].$$
We may write
\begin{align*}
\mu^{\si,x}_t(\eta,y)&=%\Pj(\tau>t)\Pj_{\si,x}\((\xi_t,X_t)=(\eta,y)\:|\: \tau>t\)+\\& \qquad +
\sum_{s=0}^t\Pj(\tau=s)\Pj_{\si,x}\((\xi_t,X_t)=(\eta,y)\:|\:\tau=s \)\\& =
(1-\a)^t\ind_\si(\eta)P_\si^t(x,y)+\sum_{s=1}^t\a(1-\a)^{t-s}\sum_{z\in[n]}\sum_{\xi\in \cC}\bu(\eta)\mu_{s-1}^{\si,x}(\xi,z)P_\eta^{t-s}(z,y).
\end{align*}
Since $\mu^{\si,x}_{s-1}(\xi,z)$ admits the same decomposition we obtain the expansion: 
\begin{align*}
\mu^{\si,x}_t(\eta,y)=A_t^{\si,x}(\eta,y)+B_t^{\si,x}(\eta,y)+C_t^{\si,x}(\eta,y),
\end{align*}
where
%do a recursion step and obtain
%$$\mu^{\si,x}_{s-1}(\xi,z)=(1-\a)^{s-1}\ind_\si(\xi)P_\si^{s-1}(x,z)+\sum_{r=1}^{s-1}\bu(\xi)\a(1-\a)^{s-1-r}\sum_{v\in[n]}\sum_{\omega\in \cC}\mu_{r-1}^{\si,x}(\omega,v)P_\eta^{s-1-r}(v,z)$$
%we can conclude that
%\begin{align*}
%\mu&^{\si,x}_t(\eta,y)=(1-\a)^t\ind_\si(\eta)P_\si^t(x,y)+\sum_{s=1}^t\bu(\eta)\a(1-\a)^{t-s}\sum_{z\in[n]}\sum_{\xi\in \cC}\times\\
%&\times\left[  (1-\a)^{s-1}\ind_\si(\xi)P_\si^{s-1}(x,z)+\sum_{r=1}^{s-1}\bu(\xi)\a(1-\a)^{s-1-r}\sum_{v\in[n]}\sum_{\omega\in \cC}\mu_{r-1}^{\si,x}(\omega,v)P_\xi^{s-1-r}(v,z)        \right]  P_\eta^{t-s}(z,y).
%\end{align*}
%Call
\begin{gather*}
A_t^{\si,x}(\eta,y)=(1-\a)^t\ind_\si(\eta)P_\si^t(x,y),\\
 B^{\si,x}_t(\eta,y)=\a(1-\a)^{t-1}\sum_{s=1}^t\sum_{z\in[n]}\bu(\eta)P_\si^{s-1}(x,z)P_\eta^{t-s}(z,y),
\\
C^{\si,x}_t(\eta,y)=\sum_{s=1}^t\sum_{r=1}^{s-1}\a^2(1-\a)^{t-1-r}
\sum_{v,z\in[n]}\sum_{\xi,\omega\in \cC}\bu(\eta)\bu(\xi)\mu_{r-1}^{\si,x}(\omega,v)P_\xi^{s-1-r}(v,z)P_\eta^{t-s}(z,y).
\end{gather*}
%
%\begin{align}
% B^{\si,x}_t(\eta,y):=\sum_{s=1}^t\bu(\eta)\a(1-\a)^{t-s}\sum_{z\in[n]}(1-\a)^{s-1}P_\si^{s-1}(x,z)P_\eta^{t-s}(z,y),
%\end{align}
%
%\begin{align}
%C^{\si,x}_t(\eta,y):=\sum_{s=1}^t\bu(\eta)\a(1-\a)^{t-s}\sum_{z\in[n]}\sum_{\xi\in \cC}\sum_{r=1}^{s-1}\bu(\xi)\a(1-\a)^{s-1-r}\sum_{v\in[n]}\sum_{\omega\in \cC}\mu_{r-1}^{\si,x}(\omega,v)P_\xi^{s-1-r}(v,z)P_\eta^{t-s}(z,y),
%\end{align}
%%so that
%\begin{align*}
%\mu^{\si,x}_t(\eta,y)=A_t^{\si,x}(\eta,y)+B_t^{\si,x}(\eta,y)+C_t^{\si,x}(\eta,y).
%\end{align*}
Notice that if $W=\nu$, or $W=B_t^{\si,x}$, or $W=C_t^{\si,x}$, for any fixed choice of $\si\in\cC$ one has
$$\sum_{\eta:\, \eta\not=\sigma}\sum_{y\in[n]}W(\eta,y)=\sum_{\eta\in\cC}\sum_{y\in[n]}W(\eta,y)+O\(|\cC|^{-1} \).$$
Therefore,
\begin{align}\label{eq:inse0}
2\|\mu_{t}^{\si,x}-\nu\|_{\tv}&=\sum_{\eta\in\cC}\sum_{y\in[n]}|\mu_{t}^{\si,x}(\eta,y)-\nu(\eta,y)|\nonumber\\&=
\sum_{y\in[n]}|\mu_{t}^{\si,x}(\si,y)-\nu(\si,y)|+ \sum_{\eta\not=\sigma}\sum_{y\in[n]}|\mu_{t}^{\si,x}(\eta,y)-\nu(\eta,y)|\nonumber\\&=
(1-\a)^t+
\sum_{\eta \in \cC}\sum_{y\in[n]}\left|B_t^{\si,x}(\eta,y)+C_t^{\si,x}(\eta,y)-\bu(\eta)\pi_\eta(y)\right|+o(1).
\end{align}
We may rewrite $C_t^{\si,x}(\eta,y)=\chi \bu(\eta) \widehat{C}^{\si,x}_t(\eta,y)$, where 
\begin{gather*}
\chi=1-(1-\a)^t-\a t(1-\a)^{t-1}=\a^2\sum_{s=1}^t\sum_{r=1}^{s-1}(1-\a)^{t-r-1},\\
\widehat{C}^{\sigma,x}_t(\eta,y)=\frac{1}{\chi}\a^2\sum_{s=1}^t\sum_{r=1}^{s-1}(1-\a)^{t-r-1}\sum_{z\in[n]}\sum_{v\in[n]}\mu_{r-1}^{\si,x}(v)\Pan_v\(X_{s-1-r}=z\) P_\eta^{t-s}(z,y),
\end{gather*}
and we use the notation $\mu_{r-1}^{\si,x}(v):=\sum_{\omega\in \cC}\mu_{r-1}^{\si,x}(\omega,v)$.  Notice that $\widehat{C}^{\sigma,x}_t(\eta,\cdot)$ is a probability on $[n]$. 
Define also the probability $\lambda_\eta$ by 
\begin{align}\label{lambdaeta}
\lambda_\eta(y)=\frac{1}{\chi}\a^2\sum_{s=1}^t\sum_{r=1}^{s-1}(1-\a)^{t-r-1}\muin P_\eta^{t-s}(y).
\end{align}
Lemma \ref{le:general} implies that uniformly in $\eta\in\cC$:
%\begin{align*}
%\| \lambda_\eta -\widehat C^{\si,x}_t(\eta,\cdot)\|_{\tv}=o(1).
%\end{align*}
%Hence
\begin{align}\label{eq:inse}
\| \widehat C^{\si,x}_t(\eta,\cdot) -\pi_\eta\|_{\tv}=\| \lambda_\eta-\pi_\eta\|_{\tv}+o(1).
\end{align}
Moreover, Lemma \ref{le:approx} implies that whenever $t-s\to\infty$:
\begin{align}\label{eq:inse2}
\sum_{\eta}\bu(\eta)\|\muin P_\eta^{t-s}-\pi_\eta\|_{\tv}=o(1).
\end{align}
Since $\alpha\to 0$, \eqref{lambdaeta} and \eqref{eq:inse2} imply
\begin{align}\label{eq:inse3}
\sum_{\eta}\bu(\eta)\|\lambda_\eta-\pi_\eta\|_{\tv}=o(1).
\end{align}
Inserting \eqref{eq:inse}, \eqref{eq:inse3} in \eqref{eq:inse0} we obtain
\begin{align}\label{eq:inse10}
2\|\mu_{t}^{\si,x}-\nu\|_{\tv}&=
(1-\a)^t+
\sum_{\eta \in \cC}\sum_{y\in[n]}\left|B_t^{\si,x}(\eta,y)+(1-\chi)\bu(\eta)\pi_\eta(y)\right|+o(1).
\end{align}
Let us now take $t=\beta\alpha^{-1}$, for some fixed constant $\beta>0$. Since $\alpha\to 0$ we have $1-\chi\to e^{-\beta}(1+\beta)$ and 
\begin{align}\label{eq:asymp}
2\|\mu_{t}^{\si,x}-\nu\|_{\tv}= e^{-\beta}+\sum_{\eta\in \cC}\bu(\eta)\psi^\si_{t}(\eta)+o(1),
\end{align}
where we define
\begin{equation}\label{eq:control}
\psi^\si_{t}(\eta)=\sum_{y\in[n]}\left|\beta e^{-\beta} \widehat{B}_t^{\si,x}(\eta,y)
-e^{-\beta}(1+\beta)\pi_\eta(y) \right|,\,
%\widehat{B}_t^{\si,x}(\eta,y):=\frac{1}{t}\sum_{s=1}^t\sum_{z\in[n]}P_\si^{s-1}(x,z)P_\eta^{t-s}(z,y).
\end{equation}
with $\widehat{B}_t^{\si,x}(\eta,\cdot)$ the probability on $[n]$ defined by
\begin{equation}\label{eq:defwidehatb}
\widehat{B}_t^{\si,x}(\eta,y)=\frac{1}{t}\sum_{s=1}^t\sum_{z\in[n]}P_\si^{s-1}(x,z)P_\eta^{t-s}(z,y).
\end{equation}
We start by noting that 
$$\psi^\si_t(\eta)\le2\beta e^{-\beta}\|\widehat{B}_t^{\si,x}(\eta,\cdot)-\pi_\eta  \|_{\tv}+e^{-\beta}.$$
In particular, \eqref{eq:asymp} shows that uniformly in $(\sigma,x)\in\cC\times[n]$
$$\|\mu_{t}^{\si,x}-\nu\|_{\tv}\le (1+\beta)e^{-\beta}+o(1),$$
which proves \eqref{eq:up-daj}.
At this point we split the analysis in four cases.
%\begin{enumerate}
%	\item  \underline{$\a T_\ent\to\gamma=\infty$.} 
%	\end{enumerate}

\subsubsection{$\alpha \tent\to\gamma=\infty$}	
In this case we notice that for fixed $\beta>0$, $$t=\beta \a^{-1}=o(\log n).$$
%	both $s-1$ and $t-s$ satisfy
%	$s-1,t-s=o(\log n)$.
	Therefore, for every $\sigma\in\cC$, $x\in[n]$ there must exist a set $\cI\subset[n]$ such that for all $1\leq s\leq t$:
	$$P_\si^{s-1}(x,\cI)=1,\qquad |\cI|\le \Delta^{t}=n^{o(1)}. $$
%	Hence,
%	$$\left\|\sum_{s=1}^t\frac{1}{t}\sum_{z\in[n]}P_\si^{s-1}(x,z)P_\eta^{t-s}(z,y)-\sum_{s=1}^t\frac{1}{t}\sum_{z\in\cI}P_\si^{s-1}(x,z)P_\eta^{t-s}(z,y)\right\|_{\tv}=o(1).$$
	Moreover, for every $\eta\in\cC$ and for every $z\in[n]$ there exists a set $\cJ_z\in[n]$ such that for every $s\le t$ 
	$$P_\eta^{t-s}(z,\cJ_z)=1,\qquad |\cJ_z|=n^{o(1)}. $$
	Therefore, 
	setting $\cJ=\cup_{z
		\in\cI}\cJ_z$, 
$$|\cJ|=n^{o(1)},\qquad \widehat B^{\si,x}_t(\eta,\cJ)=1-o(1).$$
On the other hand,  w.h.p.\ with respect to $\eta$ one has $\pi_\eta(\cJ)=o(1)$. To prove this we use the following statement, which will be established separately in Lemma \ref{le:widespread-pi} below: 
for some constant $C>0$, w.h.p.\ with respect to $\eta$ one has 
\begin{equation}\label{le41}
\sum_{x\in[n]} \pi_\eta(x)^2\leq Cn^{-1}.
\end{equation}
Assuming \eqref{le41},  for any $U\subset [n]$, the Cauchy--Schwarz inequality implies
\begin{equation}\label{eq:csc}
\pi_\eta(U)^2\leq |U|\sum_{x\in[n]} \pi_\eta(x)^2\leq C|U|n^{-1}.  
\end{equation}
 It follows that w.h.p.\  with respect to $\eta$ and uniformly  with respect to $\si$,   
\begin{align*}
\psi^\si_t(\eta)=%&2\beta e^{-\beta}\| \widehat B^{\si,x}_t(\eta,\cdot)-\pi_\eta\|_{\tv}+e^{-\beta}\sum_{y\in[n]}\pi_\eta(y)+o(1)\sim 
2\beta e^{-\beta}+e^{-\beta} +o(1).
\end{align*}
In conclusion, \eqref{eq:asymp} implies  
$$\|\mu^{\si,x}_t-\nu\|_{\tv}= (1+\beta)e^{-\beta}+o(1),$$
which proves the desired statement \eqref{scena2}.  Note that, because of the average over $\eta\in\cC$ in \eqref{eq:asymp}, the convergence in \eqref{scena2} actually holds uniformly in $\si\in\cC$ rather than in $\bbP$-probability as stated. 

\subsubsection{$\a \tent\to\gamma=0$.}
	In this case it possible to find a sequence $\upsilon=\upsilon(n)=o(1)$ that vanishes sufficiently slowly that
	$$\upsilon t=\upsilon \beta\a^{-1}\gg T_\ent.$$
	If $\widehat E_t^{\si,x}(\eta,\cdot)$ denotes the probability on $[n]$ defined by
	\begin{equation}\label{eqdefehat}
	\widehat E_t^{\si,x}(\eta,y)=\frac{1}{(1-2\upsilon)t}\sum_{s=\upsilon t}^{(1-\upsilon)t}\sum_{z\in[n]}P_{\si}^{s-1}(x,z)P_\eta^{t-s}(z,y),
	\end{equation}
	then using the definition of $\widehat B_t^{\si,x}$ in \eqref{eq:defwidehatb} we immediately get
	\begin{equation}\label{beeq}
	\|\widehat B_t^{\si,x}(\eta,\cdot)-\widehat{E}_t^{\si,x}(\eta,\cdot)\|_{\tv}=O(\upsilon).
	\end{equation}
Let us write
	$$ \sum_{z\in[n]}P_\si^{s-1}(x,z)P_\eta^{t-s}(z,\cdot)=:\lambda P_\eta^{t-s}(\cdot),$$
	which implicitly defines a probability measure $\l$, 
%	for some probability $\lambda$ on $[n]$,
and notice that % if $\eta\in\cG$ and $t-s=\omega(T_\ent)$ then for every probability distribution $\lambda$
	$$\|\lambda P_\eta^{t-s}-\pi_\eta\|_{\tv}\le \max_{x\in[n]}\| P_\eta^{t-s}(x,\cdot)-\pi_\eta\|_{\tv}.$$
	Since $\upsilon t\leq s\leq (1-\upsilon)t$, one has $t-s\gg T_\ent$, and from Theorem \ref{th:BCS} we conclude that w.h.p.\ with respect to $\eta$:
	\begin{align*}
	\|\widehat E_t^{\si,x}(\eta,\cdot)-\pi_\eta\|_{\tv}=o(1).
	\end{align*}
	Therefore, w.h.p.\
	$$	\|\widehat B_t^{\si,x}(\eta,\cdot)-\pi_\eta\|_{\tv}=o(1).$$
	By the triangular inequality and \eqref{eq:asymp},
	\begin{align*}
	\|\mu_{t}^{\si,x}-\nu\|_{\tv}=e^{-\beta} +o(1).
	\end{align*}
This proves \eqref{scena1}. As in the previous case, it is worth noting that the convergence in \eqref{scena1} actually holds uniformly in $\si\in\cC$ rather than in $\bbP$-probability.

\subsubsection{$\a T_\ent\to\gamma\in(0,\infty)$ and $\beta<\gamma$} \label{subsec222}
	We want to control $\psi^\si_t(\eta)$ as defined in \eqref{eq:control}. 	If $\beta<\gamma$ then $t=(1-\epsilon)T_\ent$ for some $\epsilon\in(0,1)$. 
%	For every $\upsilon\in(0,\epsilon/2)$, at the price of an additive error $O(\upsilon)$ in total variation,  we can 
%	replace $\widehat{B}_t^
%	{\si,x}(\eta,\cdot)$ by the probability $\widehat{E}(\cdot)$ defined as 
%%	express the probability measure $\widehat{B}_t^
%%	{\si,x}(\eta,\cdot)$ as the convex combination of three probability measures
%%	\begin{equation}\label{eq:dec}
%%	 \widehat{B}_t^
%%	{\si,x}(\eta,\cdot)=b_1\widehat{E}(\cdot)+b_2\widehat{B}_2(\cdot)+b_3\widehat{B}_3(\cdot)
%%	\end{equation}
%%	where
%	$$\widehat{E}(y)=\frac{1}{(1-2\upsilon)t}\sum_{s=\upsilon t}^{(1-\upsilon)t} \sum_{z\in[n]}P_{\si}^{s-1}(x,z)P_\eta^{t-s}(z,y).$$
%	Hence, taking $\upsilon\to 0 $ 
%	\begin{align}\label{eq:control2}
%	\psi^\si_t(\eta)=\sum_{y\in[n]}\left|\beta e^{-\beta} \widehat{E}(y) -e^{-\beta}(1+\beta)\pi_\eta(y)\right|+ o(1).
%	\end{align}
%	\begin{color}{red}
%For simplicity we write $\widehat B$ for $	\widehat{B}_t^{\si,x}(\eta,\cdot)$.
We argue that w.h.p.\ with respect to the independent pair $(\sigma,\eta)$, uniformly in $x$,
	\begin{equation}
	\label{eq:csc1}
	\|\widehat{B}_t^{\si,x}(\eta,\cdot)-\pi_\eta \|_{\tv}=1-o(1).
	\end{equation}
Let $U_x^{s-1}$ denote the set of $y\in[n]$ such that
	$$\sum_{z\in[n]}P^{s-1}_\si(x,z)P^{t-s}_\eta(z,y)\ge e^{-(1+\epsilon)H t}$$
	Summing over $y\in U_x^{s-1}$, we must have
	$$|U_x^{s-1}|\le e^{(1+\epsilon)H t} %\leq n^{(1+\epsilon)(1-\epsilon)}
	=n^{1-\epsilon^2}$$
%	Choose $\e=\epsilon$, so that
%	$$|Y_{s,t}|\le n^{1-\epsilon^2}.$$
In the proof of Theorem \ref{th:2cutoff} , see Eq. \eqref{eq:notas} and \eqref{quxs} below, we prove that  w.h.p.\ with respect to  $(\si,\eta)$
$$	\min_{s\in[1,t]}\min_{x\in[n]}\sum_{y\in U_x^{s-1}}\sum_{z\in[n]}P^{s-1}_\si(x,z)P^{t-s}_\eta(z,y)=1-o(1).$$
	Setting $U=\cup_{s=1}^{t}U_x^{s-1}$ 	we have 
	$$\widehat{B}_t^{\si,x}(\eta,U)=\sum_{y\in U}\widehat{B}_t^{\si,x}(\eta,y)\ge \frac{1}{t}\sum_{s=1}^{t}\sum_{y\in U_x^{s-1}}\sum_{z\in[n]}P^{s-1}_\si(x,z)P^{t-s}_\eta(z,y)=1-o(1).$$
Since $|U|\le \tent\,n^{1-\epsilon^2}=o(n)$,	$\widehat{B}_t^{\si,x}(\eta,\cdot)$ is w.h.p.\ asymptotically  singular with respect to $\pi_\eta$, namely 
$$
\|\widehat{B}_t^{\si,x}(\eta,\cdot)-\pi_\eta \|_{\tv}\geq \widehat{B}_t^{\si,x}(\eta,U) - \pi_\eta(U) =1 -o(1),
$$
where we have estimated $\pi_\eta(U) =o(1)$ as in  \eqref{eq:csc}.
This proves \eqref{eq:csc1}.
%	We conclude that \eqref{eq:control2} converges to $(1+2\beta) e^{-\beta}$.	
Inserting this in \eqref{eq:asymp}--\eqref{eq:control}, it follows that w.h.p.\ with respect to $\si\in\cC$:
	\begin{equation*}
	\|\mu^{\si,x}_t-\nu\|_{\tv}=(1+\beta)e^{-\beta}+o(1).
	\end{equation*}

\subsubsection{$\a T_\ent\to\gamma\in(0,\infty)$ and $\beta>\gamma$.} By definition there must exist some $\epsilon>0$ such that
$$t=\beta \a^{-1}=(1+o(1))\frac{\beta}{H\gamma}\log n>(1+\epsilon) T_\ent,$$
for all $n$ large enough.
As in \eqref{beeq}, for every $\upsilon\in(0,\epsilon/2)$, at the price of an additive error $O(\upsilon)$ in total variation,  we can 
	replace $\widehat{B}_t^
	{\si,x}(\eta,\cdot)$ by the probability $\widehat{E}^{\si,x}_t(\eta,\cdot)$ defined in \eqref{eqdefehat}. 
%	$$\widehat{E}(y)=\frac{1}{(1-2\upsilon)t}\sum_{s=\upsilon t}^{(1-\upsilon)t} \sum_{z\in[n]}P_{\si}^{s-1}(x,z)P_\eta^{t-s}(z,y).$$
Since $t>(1+\epsilon) T_\ent$ and the summation involves only $s\in[\upsilon t,(1-\upsilon)t]$, we can use Theorem \ref{th:2cutoff} to obtain that w.h.p.\ with respect to the independent pair $(\si,\eta)$,
	\begin{equation}\label{eq:bigsmall}
	\| \widehat{E}^{\si,x}_t(\eta,\cdot)-\pi_\eta\|_{\tv}=o(1)
	\end{equation}
From \eqref{eq:control}, 
\begin{align*}
\psi^\si_t(\eta)=e^{-\beta}\sum_{y\in[n]}|\beta \widehat{E}^{\si,x}_t(\eta,y)-(1+\beta)\pi_\eta(y) |+O(\upsilon)
%=&e^{-\beta}\sum_{y\in[n]}\left|\beta b_1 \widehat{E}(y)-(1+b_1\g)\pi_\eta(y) \right|+O(\upsilon)\\
=e^{-\beta}+O(\upsilon)+o(1).%=e^{-\beta}[2\gamma+1]+O(\upsilon).
\end{align*}
Since $\upsilon$ is arbitrarily small, from \eqref{eq:asymp} we obtain that w.h.p.\ with respect to $\si\in\cC$:
$$\|\mu^{\si,x}_t-\nu\|_{\tv}=e^{-\beta}+o(1).$$

\section{Trichotomy for the random walk}\label{trichRW}
Here we prove Theorem \ref{th:trichotomy-rw}. The main observation can be stated as follows. 
\begin{proposition}\label{pr:first-refresh}
Let $\t=\t(t)$ denote the random variable in  \eqref{eq:rts}. Then, uniformly in $t\geq 2$:
	$$\lim_{n\to\infty}\max_{\sigma,x}\| \Pj_{\sigma,x}\(X_t=\cdot\:|\:1\leq \tau<t \)-\muin\|_{\tv}=0$$
\end{proposition}
\begin{proof}
Observe that 
\begin{align}\label{eq:ob}
\Pj_{\sigma,x}\(\tau\in\{0,t\} \)=1-\Pj_{\sigma,x}\(1\leq \tau<t \) = (1-\alpha)^t + \a(1-\a)^{t-1} = (1-\a)^{t-1}.
\end{align}
Moreover, if $1\leq s<t$, %$x,y\in[n] $,
	$$\Pj_{\sigma,x}\(X_t=y;\:\tau=s \)=\alpha(1-\alpha)^{t-s}\sum_{\eta\in \cC}\bu(\eta)\lambda P_\eta^{t-s}(y)$$
	where $\lambda$ is the probability measure given by 
	$$\lambda(z)=\Pj_{\sigma,x}(X_{s}=z\:|\:\tau=s).$$
	We then compute the conditional probability
	\begin{align}
	\Pj_{\sigma,x}\(X_t=y\:|\:1\leq \tau<t \)=&\frac{1}{1-(1-\alpha)^{t-1}}\sum_{s=1}^{t-1}\Pj_{\sigma,x}\(X_t=y;\:\tau=s \)\\
	=&\frac{1}{1-(1-\alpha)^{t-1}}\sum_{s=1}^{t-1}\alpha(1-\alpha)^{t-s}\sum_{\eta\in \cC}\bu(\eta)\lambda P_\eta^{t-s}(y).
	\end{align}
	Now we can rely on the uniform bound of Lemma \ref{le:general} to conclude
	\begin{align*}
	&\| \Pj_{\sigma,x}\(X_t=\cdot\:|\:1\leq \tau<t \)-\muin\|_{\tv}
	\\
	&\qquad \le\sum_{s=1}^{t-1}\frac{\a(1-\a)^{t-s}}{1-(1-\a)^{t-1}}\max_{x\in[n]}\| \Pan_x(X_{t-s}=\cdot)-\muin\|_{\tv}=o(1).
	\qedhere\end{align*}
\end{proof}

\begin{corollary}\label{co:trichotomy-rw}
Uniformly in $t\geq 1$:
	\begin{align*}
	\|\Pj_{\sigma,x}(X_t=\cdot)-\muin\|_{\tv}=(1-\alpha)^{t}\|P_\sigma^t(x,\cdot)-\muin\|_{\tv}+o(1).
	\end{align*}
\end{corollary}
\begin{proof} 
Note that 
$$\Pj_{\sigma,x}(X_t=\cdot\:|\:\tau=0)=P_\sigma^t(x,\cdot)\,,\qquad \Pj_{\sigma,x}(X_t=\cdot\:|\:\tau=t)=P_\sigma^{t-1}(x,\cdot).$$
Using Proposition \ref{pr:first-refresh}, and $\a\to 0$, the triangle inequality shows that
	\begin{align*}
&	\|\Pj_{\sigma,x}(X_t=\cdot)-\muin\|_{\tv}\\&\quad =	
	\|(1-\a)^{t-1}\Pj_{\sigma,x}(X_t=\cdot\:|\:\tau\in\{0,t\})+(1-(1-\a)^{t-1})\Pj_{\sigma,x}(X_t=\cdot\:|\:1\leq \tau<t )-\muin\|_{\tv}\\&\quad
	%\quad =	\|(1-\a)^{t-1}\Pj_{\sigma,x}(X_t=\cdot\:|\:\tau\in\{0,t\})+\Pj_{\sigma,x}\(1\leq \tau<t \)\Pj_{\sigma,x}(X_t=\cdot\:|\:1\leq \tau<t )-\muin\|_{\tv}\\&\quad
	=(1-\alpha)^{t-1}\|\Pj_{\sigma,x}(X_t=\cdot\:|\:\tau\in\{0,t\})-\muin\|_{\tv}+o(1)\\& \quad
	= (1-\alpha)^{t-1}\|(1-\alpha)P_\sigma^t(x,\cdot) + \alpha P_\sigma^{t-1}(x,\cdot)-\muin\|_{\tv}+o(1)
	\\& \quad
	= (1-\alpha)^{t}\|P_\sigma^t(x,\cdot) -\muin\|_{\tv}+o(1).
	\qedhere\end{align*}
	\end{proof}
\bigskip

All statements in Theorem \ref{th:trichotomy-rw} follow from Corollary \ref{co:trichotomy-rw} provided we establish the next lemma. 
\begin{lemma}\label{le:1-eps}
	If $t=\beta T_\ent$ then for any fixed $\beta>0$,
	$\beta\neq 1$:
	\begin{equation}
			\max_{x\in[n]}\left|\|P_\sigma^t(x,\cdot) -\muin\|_{\tv}-\varphi(\beta )\right|\overset{\P}{\longrightarrow}0,
			\end{equation}
			where $\varphi(\beta)=1$ if $\beta<1$ and $\varphi(\beta)=q$ for $\beta>1$, and $q$ is defined in \eqref{eq:dista}. 
%	 then, fixed any arbitrarily small $\varepsilon>0$ we have that for every $\eta\in\cG$
%	$$\min_{x\in[n]}\left\|P^s_\eta(x,\cdot)-\muin\right\|_{\tv}\ge 1-\varepsilon.$$
\end{lemma}
\begin{proof}
 From Theorem \ref{th:BCS} it is sufficient to show that 
% for every $x\in[n]$, 
if $t= \beta\tent$ with $\beta<1$, then 
  for any $\e>0$ w.h.p.\
  \begin{gather}\label{eq:gath}
  \min_{x\in[n]}\left\|P^t_\eta(x,\cdot)-\muin\right\|_{\tv}\ge 1-\varepsilon,
  \end{gather}
  and that 
   \begin{gather}\label{eq:gath2}
q- \|\pi_\eta-\muin\|_{\tv}
 \overset{\P}{\longrightarrow}0.
  \end{gather}
  The concentration \eqref{eq:gath2} has been already proved in \cite[Lemma 17]{BCS1} for the DCM. A similar proof applies to the OCM. %(see also \cite[Proposition 6]{CQ:PageRank}). 
Concerning the estimate \eqref{eq:gath}, we can use Lemma \ref{le:lln} below (taking $s=0$) to show that if $t= \beta\tent$ with $\beta<1$ then there exists a set $U_x\subset[n]$ with $|U_x|=o(n)$ such that w.h.p.\
 $$P_\eta^{t}(x,U_x)\ge 1- o(1),$$
see  \eqref{quxs} for more details. Since $\muin(U_x)=o(1)$, this ends the proof. 
\end{proof}

\section{Cutoff in double digraphs}\label{sec:2cutoff}
We start by showing that 
%two lemmas for the case of a single random environment. The combination of the two implies that 
w.h.p.\ the stationary distribution of a random digraph in any of the two models is a widespread measure.
	
	\subsection{The stationary distribution is widespread}
	\begin{lemma}\label{le:widespread-pi}
	There exists a constant $C= C(\Delta)>0$ such that
		$$\lim_{n\to\infty}\P\(n\sum_{z\in[n]} \pi_\si(z)^2\le C \)=1.$$
	\end{lemma}
	\begin{proof}
		Call $Z=n\sum_{z\in[n]} \pi_\si(z)^2$. Let $t=\log^3(n)$ and consider the event 
		\begin{equation}\label{eq:def-D}
		\cD=\left\{\si\in\cC:\,\max_{x,z\in[n]}|\pi_\si(z)-P_\si^t(x,z)|=o(n^{-3})\right\}.
		\end{equation}
		A simple consequence of Theorem \ref{th:BCS} (see \cite[Lemma 3.11]{CQ:cover}) is that $\P(\cD)=1-o(1)$. Therefore,
		$$\P(Z>C)= \P(Z>C;\cD)+o(1)%=\P\(Z\ind_\cD>C \)+o(1)
		.$$
		By Markov's inequality
		$$\P\(Z>C;\:\cD \)\le \frac{\E[Z^K\ind_\cD]}{C^K},\qquad\forall K\ge 1.$$
		Therefore, it is sufficient to show that $\E[Z^K\ind_{\cD}]\le (C/2)^K$ for some $K=K(n)\to\infty$. Choose for example $K=\log n$. Then,
		\begin{align*}
		\E[Z^K\ind_\cD]
		\le& n^K\E\[\(\sum_{z\in[n]}\sum_{x\in[n]}\sum_{y\in[n]}\frac{1}{n^2}\(P^t_\si(x,z)+o(n^{-3})\)\(P^t_\si(y,z)+o(n^{-3})\)\)^K\]\\
		\le& n^K\(\E\[\(o(n^{-1})+\sum_{z\in[n]}\sum_{x\in[n]}\sum_{y\in[n]}\frac{1}{n^2}P^t_\si(x,z)P^t_\si(y,z)\)^K\] \)\\
		\le& (2n)^K\(\E\[\(\sum_{z\in[n]}\sum_{x\in[n]}\sum_{y\in[n]}\frac{1}{n^2}P^t_\si(x,z)P^t_\si(y,z)\)^K\] \)+o(1)\\
		=&(2n)^K\Pan_{ unif}\(X^{(\ell)}_t=Y^{(\ell)}_t,\:\forall \ell\le K\)+o(1),
		\end{align*}
%		\begin{comment}
%		\E[Z^K\ind_\cD]
%		\le& (2n)^K\E\[\(\sum_{z\in[n]}\sum_{x\in[n]}\sum_{y\in[n]}\frac{1}{n^2}P^t_\si(x,z)P^t_\si(y,z)\)^K\]+2^K\\
%		=&(2n)^K\Pan_{ unif}\(X^{(\ell)}_t=Y^{(\ell)}_t,\:\forall \ell\le K\)+2^K
%		\end{comment}
		where $\Pan_{unif}$ denotes the annealed law of the $2K$ independent walks $(X^{(k)}_s,Y^{(k)}_s)_{ s\le t}$ for $ k\le K$, each starting at a uniformly random vertex:
		\begin{equation}\label{eq:anno}
\Pan_{unif}=\frac{1}{n^{2K}}\sum_{x_1,\dots,x_K}\sum_{y_1,\dots,y_K}\sum_{\eta\in\cC}\bu(\eta) \mathbf{P}^\eta_{x_1}\cdots\mathbf{P}^\eta_{x_K}\mathbf{P}^\eta_{y_1}\cdots\mathbf{P}^\eta_{y_K}.
\end{equation}	
% The $2K$ walks are independent conditionally on the environment, and the average is both over the walks and the environment. 
 For an explicit construction, we can generate recursively the walks and the environment, letting the trajectories reveal the configuration $\eta$, the $\ell$-th trajectory living in the environment discovered by the previous $\ell-1$ trajectories; see \cite[Lemma 3.11]{CQ:cover} for  more details.
		Therefore, it is sufficient to show that it is possible to find a constant $C>0$ such that for every sufficiently large $n$
		\begin{equation}\label{eq:pan}
		\Pan_{ unif}\(X^{(\ell)}_t=Y^{(\ell)}_t,\:\forall \ell\le K\)\le \(\frac{C}{4n}\)^K.
		\end{equation}
		For $k=1,\dots,K$, define the events
		\begin{align}\label{eq:def-bk}
		B_k= \bigcap_{\ell\le k} \{ X_t^{(\ell)}=Y_t^{(\ell)} \},
		\end{align}
		and call $A_k$ the set of vertices which have at least one tail/head revealed by the trajectories $\(X^{(\ell)},Y^{(\ell)}\)_{\ell\le k}$.
	Let $\Xi_k$ denote a realization of the trajectories $(X^{(\ell)}_s,Y^{(\ell)}_s)_{s\le t, \ell\le k}$ satisfying $B_k$. We are going to prove that, uniformly in the realization $\Xi_k$ and $k\leq K$, % and consider the conditional probability
		\begin{align}\label{eq:bk}
		\Pan_{unif}(B_{k+1}\:|\: \Xi_k)
		=&\sum_{z\in[n]}\sum_{x\in[n]}\sum_{y\in[n]}\frac{1}{n^2}\Pan_{x,y}\(X^{(k+1)}_t=Y^{(k+1)}_t=z \:|\: \Xi_k\) = O(1/n).
		\end{align}

		\begin{comment}
			\item 		We start by showing that the terms in \eqref{eq:secondcase}-\eqref{eq:lastcase} are $o(n^{-1})$. In fact
			\item \eqref{eq:secondcase} :
			\item \eqref{eq49}: we have that fro every $z\in A_k^c$
			$$\P_{z}^{an}\(X_t^{(k+1)}=z\:|\:B_k,\cS_k \)=O(Kt^2/n),$$
			hence
			$$\eqref{eq49}\le O\( \frac{|A_k^c|K t^2}{n^3} \)=o(n^{-1}).$$
			\item \eqref{eq410}: we have that fro every $z,x\in A_k^c$ with $x\not=z$
			$$\P_{x,x}^{an}\(X_t^{(k+1)}=Y_t^{(k+1)}=z\:|\:B_k,\cS_k \)=...$$
			hence
			\item \eqref{eq411}:
			\item \eqref{eq412}:
				\item \eqref{eq412}:
			\item \eqref{eq:lastcase}: Trivially
			$$ \Pan_{x,y}\(X^{(k+1)}_t=Y^{(k+1)}_t \:|\: B_k,\cS_k\)\le 1,$$
			hence
			$$  \eqref{eq:lastcase}\le O\(\frac{|A_k|^2}{n^2}\)=o(n^{-1}).$$
		\end{comment}
		We first show that if $x$, $y$, $z$ are three distinct vertices all in $A_k^c=[n]\setminus A_k$ then, uniformly in $\Xi_k$, 
		\begin{align}\label{eq:akk}
		\Pan_{x,y}\(X^{(k+1)}_t=Y^{(k+1)}_t=z\:\rvert\Xi_k\)=O\(\frac{1}{n^2}\).
		\end{align}
		Consider the event $\cE_k$ that the trajectory $X^{(k)}=\{X_s^{(k)},\,0\le s\le t\}$ has no collision with itself nor with the environment previously discovered by $X^{(1)},Y^{(1)}\dots,X^{(k-1)},Y^{(k-1)}$, and let $\cY_k$ denote the event that the walk $X^{(k)}$ does not visit $y$. At any given time, any given walk has probability at most $\Delta/(m-Kt)=O(1/n)$ of hitting a given vertex by generating a fresh new edge.   Thus, by a union bound, the event $\cE_k^c\cup\cY_k^c $ has probability $O(Kt^2/n)$ uniformly in $k\le K$.
		
		We prove \eqref{eq:akk} by decomposing the event $X^{(k+1)}_t=Y^{(k+1)}_t=z$ along the four cases: $\cE_k\cap\cY_k,\cE_k^c\cap\cY_k,\cE_k\cap\cY_k^c, \cE_k^c\cap\cY_k^c$. 
%		\begin{multline}
%		\Pan_{x,y}\(X^{(k+1)}_t=Y^{(k+1)}_t=z \:|\: \Xi_k\)=\Pan_{x,y}\(X^{(k+1)}_t=Y^{(k+1)}_t=z; \cE_{k+1};\cY_{k+1} \:|\:\Xi_k\)\\ 
%		+\Pan_{x,y}\(X^{(k+1)}_t=Y^{(k+1)}_t=z;\cE_{k+1}^c;\cY_{k+1} \:|\: \Xi_k\) + \Pan_{x,y}\(X^{(k+1)}_t=Y^{(k+1)}_t=z; \cE_{k+1};\cY_{k+1}^c \:|\:\Xi_k\)
%		\\ +\Pan_{x,y}\(X^{(k+1)}_t=Y^{(k+1)}_t=z;\cE_{k+1}^c;\cY_{k+1}^c \:|\: \Xi_k\).
%		\end{multline}
	%	\begin{itemize}
		%	\item 
		Consider first the case $\cE_k^c\cap\cY_k^c$. 
%		where  $X^{(k+1)}$, before  arriving in $z$ at time $t$,   passes through $y$ and visits at least once an already discovered vertex. 
						 The probability of $\cE_k^c$ cannot exceed $O(Kt^2/n)$. Moreover, the probability of visiting $y\in A_k^c$ and $z\in A_k^c$ can each be bounded by $O(t/n)$. Thus, 			 $$\Pan_{x,y}\(X^{(k+1)}_t=z;\cE_{k+1}^c;\cY_{k+1}^c \:|\: \Xi_k\)=O(t/n)O(t/n)O(Kt^2/n)=o(n^{-2}).$$
		%	\item 
		Similarly, %we can show that the following bound holds: 
			$$\Pan_{x,y}\(X^{(k+1)}_t=Y_t^{(k+1)}=z;\cE_{k+1}^c;\cY_{k+1} \:|\: \Xi_k\)=O(Kt^2/n)O(t/n)O(Kt^2/n)=o(n^{-2}).$$
			Indeed, the walk $X^{(k+1)}$ must visit $z\in A_k^c$ and also one of the previously discovered vertices, which has probability $O(Kt^2/n)\times O(t/n)$, and, if $\cY_{k+1}$ holds, in order for the walk $Y^{(k+1)}$ to arrive in $z$ at time $t$ it is necessary to visit a vertex that was already discovered (e.g., $z$ itself). The latter event has probability $O\(Kt^2/n \)$.
		%	\item 
		
		If  $\cE_{k+1}\cap\cY_{k+1}$ holds, then for $X^{(k+1)}_t = z,Y^{(k+1)}_t = z$ to occur there must be a time $s\le t$ such that $Y^{(k+1)}$ collides at time $s$ with the trajectory of $X^{(k+1)}$, then $Y^{(k+1)}$ stays on this trajectory for $t-s$ units of time, and then finally hits $z$ at time $t$. On the event $\cE_{k+1}$ the probability of $X^{(k+1)}_t = z$ is bounded by $\frac{d_z^-}{m}(1+o(1))$, and the event that $Y^{(k+1)}$ spends $h$ units of time in the path $X^{(k+1)}$ is at most $2^{-h}$. Therefore,
			\begin{align*}\label{eq:constantC'}
			\Pan_{x,y}\(X^{(k+1)}_t=Y^{(k+1)}_t=z;\cE_{k+1};\cY_{k+1}\:\rvert\:\Xi_k\)\le\frac{d_z^-}{m}(1+o(1))\cdot \frac{2\Delta}{m}\sum_{h=1}^t 2^{-h}\le \frac{\Delta^2}{n^2}.
			\end{align*}
		%	\item  
		
	Finally, if $x\not=y$, then the event $\cE_{k+1}\cap \{ X^{(k+1)}_t=z \}\cap \cY_{k+1}^c$ has probability $O(t/n)\times O(1/n)$.  Under this event, when the walk $Y^{(k+1)}$ starts at $y$ the revealed in-neighborhood of $z$ consist of a unique path of length $t$ from $x$ to $z$ and $y$ is a vertex in this path. Since $y\neq x$, %in order for the walk 
		to achieve $Y^{(k+1)}=z$ it is necessary that $Y^{(k+1)}$ exits and re-enters the path. 
		This requires hitting the path by creating a fresh edge, which has probability  $O(Kt^2/n)$. 
Hence,			$$\Pan_{x,y}\(X^{(k+1)}_t=Y^{(k+1)}_t=z;\cE_{k+1};\cY^c_{k+1} \:|\: \Xi_k\)=O(t/n)O(1/n)O(Kt^2/n)=o(n^{-2}).$$
		%\end{itemize}
		In conclusion, we have proved \eqref{eq:akk}. In particular, we have obtained %Collecting the above estimates:
		\begin{equation}\label{eq:boundan1}
	\sum_{z\in A^c_k}\sum_{x\in A^c_k\setminus z}\sum_{y\in A^c_k\setminus z,x}\frac{1}{n^2}\Pan_{x,y}\(X^{(k+1)}_t=Y^{(k+1)}_t=z \:|\:\Xi_k\)
	\le  n^3\frac{1}{n^2}\frac{\Delta^2}{n^2}= \frac{\Delta^2}{n}.
		\end{equation}
		We now deal with the probability
		$$\Pan_{x,y}\(X^{(k+1)}_t=Y^{(k+1)}_t\:|\:\Xi_k \), $$
		when $x\in A_k$ and $y\in A_k^c$ or vice versa. By symmetry we can restrict to the former case.
		  We observe that 
		%Notice that in the former case
		\begin{equation}\label{412}
		\Pan_{x,y}\(X^{(k+1)}_t=Y^{(k+1)}_t \:|\:\Xi_k\) =O\( \frac{Kt^2}{n}\).
		\end{equation}
		Indeed, 
		\begin{equation}
		\Pan_{x,y}\(X^{(k+1)}_t=Y^{(k+1)}_t;\cY_{k+1} \:|\:\Xi_k\) =O\( \frac{Kt^2}{n}\),
		\end{equation}
		since the latter event requires that the walk $Y^{(k+1)}_t$ visits a vertex that has been already discovered by $X^{(1)},Y^{(1)},\dots,X^{(k)},Y^{(k)},X^{(k+1)}$, while
		\begin{equation}
		\Pan_{x,y}\(\cY_{k+1}^c\:|\:\Xi_k \) =O\( \frac{t}{n}\).
		\end{equation}
			Hence, using $|A_k|\leq Kt$,
			\begin{equation}\label{eq:bound2an}
			\sum_{x\in A_k}\sum_{y\in A_k^c}\frac{1}{n^2}\Pan_{x,y}\(X_t^{(k+1)}=Y_t^{(k+1)}\:|\:\Xi_k \)\le \frac{Kt\cdot n }{n^2}\times O\(\frac{K t^2}{n}\)=o(n^{-1}).
			\end{equation}
%		
%		Similarly, if we assume that $x\in A_k^c$ and $y\in A_k$ we have
%			\begin{equation}
%			\Pan_{x,y}\(X^{(k+1)}=Y^{(k+1)} \:|\:\Xi_k\) =O\( \frac{Kt^2}{n}\).
%			\end{equation}
%			In fact
%				\begin{equation}
%				\Pan_{x,y}\(X^{(k+1)}=Y^{(k+1)};\cE_{k+1} \:|\:\Xi_k\) =O\( \frac{Kt^2}{n}\),
%				\end{equation}
%				since under $\cE_{k+1}$ we have that $X_t^{(k+1)}\in A_k^c$ and when the walk $Y^{(k+1)}$ starts, there are no paths revealed from $y$ to $X^{(k+1)}_t$. On the other hand
%				\begin{equation}
%				\Pan_{x,y}\(\cE^c_{k+1} \:|\:\Xi_k\) =O\( \frac{Kt^2}{n}\).
%				\end{equation}
%				Hence
%				\begin{equation}\label{eq:bound2anbis}
%				\sum_{x\in A_k^c}\sum_{y\in A_k}\frac{1}{n^2}\Pan_{x,y}\(X_t^{(k+1)}=Y_t^{(k+1)}\:|\:\Xi_k \)\le \frac{ n\cdot Kt }{n^2}\times O\(\frac{K t^2}{n}\)=o(n^{-1}).
%				\end{equation}
	The case $x=y$ in \eqref{eq:bk} is handled by the obvious bound 
	$$\Pan_{x,y}\(X^{(k+1)}_t=Y^{(k+1)}_t\:|\: \Xi_k \)\le1.$$		
	The same can be done for the case %We are left with considering two cases: 
	$x\in A_k$ and $y\in A_k$.
 Indeed,	$|A_k|\leq Kt$ implies 
				\begin{equation}\label{eq:bound3an}
				\sum_{x\in A_k}\sum_{y\in A_k}\frac{1}{n^2}\Pan_{x,y}\(X^{(k+1)}_t=Y^{(k+1)}_t\:|\: \Xi_k \)\le \frac{K^2t^2}{n^2}=o(n^{-1}).
				\end{equation}
				Finally, if $z\in A_k$ and $x,y\in A_k^c$, $x\neq y$ then %on one hand
			\begin{equation*}
			\Pan_{x,y}\(X^{(k+1)}_t=Y^{(k+1)}_t\in A_k;\:\cY_{k+1}\:|\: \Xi_k \)=O\(\frac{K^2t^4}{n^2} \),
			\end{equation*}
			since both walks have to hit the cluster $A_k$ in order to visit $z$. On the other hand 
			\begin{equation*}
			\Pan_{x,y}\(X^{(k+1)}_t=Y^{(k+1)}_t\in A_k;\:\cY_{k+1}^c\:|\: \Xi_k \)=O\(\frac{Kt^2}{n}\)O\(\frac{t}{n} \),
			\end{equation*}
			since $X^{(k+1)}$ needs to visit both  $y$ and the cluster $A_k$. Hence
			\begin{equation}\label{eq:bound4an}
			\sum_{x\in A_k^c}\sum_{y\in A_k^c}\frac{1}{n^2}\Pan_{x,y}\(X^{(k+1)}_t=Y^{(k+1)}_t\in A_k\:|\:\Xi_k \)=O\(\frac{K^2t^4}{n^2} \)=o(n^{-1}).
			\end{equation}
		%In order for the event $\cE_k(X)^c$ to occur it is necessary to choose an head of a discovered vertex and a head of $z$. Independently on the order, such event must have probability at most $O\(Kt^2/n^2\)$. If $\cE_k(X)^c$ occurs, then the conditional probability of $Y^{(k+1)}_t = z$ is at most the probability that $Y^{(k+1)}$ collides with the first trajectory at some time $s\le t$, that is $O(Kt/m)$. Hence,
		%	\begin{align*}
		%	\Pan_{x,y}\(X^{(k+1)}_t=Y^{(k+1)}_t=z;\cE_K(X)^c\:\bigg\rvert\:B_k\)=O\(\frac{Kt^2}{n^2}\)\cdot O\(\frac{Kt}{n} \)=O\(\frac{1}{n^2} \).
		%	\end{align*}
		Therefore, putting together the bounds \eqref{eq:boundan1}, \eqref{eq:bound2an}, %\eqref{eq:bound2anbis}, 
		\eqref{eq:bound3an} and \eqref{eq:bound4an}, and recaling  \eqref{eq:bk}, 
		 we have shown that %uniformly in $k\le K$ and $\Xi_k$
		\begin{align}\label{eq:bk2}
		\nonumber\Pan_{unif}(B_{k+1} \:|\:\Xi_k)	\le\frac{\Delta^2}{n}+ o\(n^{-1}\)\le \frac{2 \Delta^2}{n}.
		\end{align}
		The same proof shows that 
		$$
		\Pan_{unif}\(B_1\) \leq \frac{2 \Delta^2}{n}.
		$$
		By the uniformity in $k\le K$ and in $\Xi_k$ of the previous argument, we conclude that
		$$\Pan_{unif}\(B_K \)=\Pan_{unif}\(B_1\)\prod_{k=1}^{K-1}\Pan_{unif}\(B_{k+1}\:|\: B_k \)\le \(\frac{2\Delta^2}{n}\)^K.$$
		Therefore it is sufficient to choose e.g.\ $C=8\Delta^2$ to conclude that \eqref{eq:pan} holds.	
	\end{proof}

	%\begin{color}{red}
	\begin{lemma}\label{le:max-pi} We have %For every $\varepsilon>0$ holds
\begin{equation}\label{eq:422}
	\lim_{n\to\infty}\P\(\max_{z\in[n]}\pi_\si(z)\le\frac{\log^8(n)}{n} \)=1.		\end{equation}
	\end{lemma}
	\begin{proof}
		For the DCM ensemble we may refer to \cite[Theorem 1.5]{CQ:cover} for a much more precise result, where $8$ is replaced by a constant $a\in[0,1]$. We give here an alternative proof of the weaker bound \eqref{eq:422} that holds for the OCM as well. We show that if $t=\log^3(n)$, then
		uniformly in $z\in[n]$
		\begin{equation}\label{eq:514}
		\P\(\sum_{x\in[n]}\frac{1}{n}P^t_\si(x,z)\ge\frac{\log^8(n)}{2n} \)=o(n^{-1}).
		\end{equation}
%		so that
%		\begin{equation}\label{eq:422}
%		\P\(\max_{z\in[n]}\sum_{x\in[n]}\frac{1}{n}P^t_\si(x,z)\le \frac{\log^8(n)}{n} \)=1-o(1).
%		\end{equation}
By the union bound, and the fact that the event $\cD$ in \eqref{eq:def-D} occurs w.h.p., \eqref{eq:514} is sufficient to prove \eqref{eq:422}.	
		Define
		$$ W:= \sum_{x\in[n]}\frac{1}{n}P^t_\si(x,z). $$
		By Markov inequality, for every $K\ge 1$
		$$\P\(W\ge \frac{\log^8(n)}{2n}  \)\le \frac{2^Kn^K}{\log^{8K}(n)}\E\left[W^K\right].$$
		As in the proof of Lemma \ref{le:widespread-pi}, the term in the right hand side of the latter display can be read in terms of the annealed walks. In conclusion, to prove \eqref{eq:514} it is sufficient to show that for $K=\log n$
		\begin{equation}\label{eq:bound-an}
		\E \left[W^K\right]=\Pan_{unif}\(X^{(k)}_t=z,\:\forall k\le K \)\le\(\frac{C\log^7(n)}{n}\)^K,
		\end{equation}
for some constant $C>0$,		where 
$\Pan_{unif}$ denotes the annealed law of $K$ independent walks 
		\begin{equation}\label{eq:annoa}
\Pan_{unif}=\frac{1}{n^{K}}\sum_{x_1,\dots,x_K}\sum_{\eta\in\cC}\bu(\eta) \mathbf{P}^\eta_{x_1}\cdots\mathbf{P}^\eta_{x_K}.
\end{equation}	
%
%
%the $(X^{(k)}_s)_{s\le t}$, for $k\in\{1,\dots, K\}$, are $K$ independent walks conditionally on the environment and the average is both over the walks and the environment.
		%At this point we adopt the strategy
		 Reasoning as in the proof of Lemma \ref{le:widespread-pi}, similarly to \eqref{eq:def-bk} we call $$B_k=\bigcap_{\ell\le k}\{X^{(\ell)}_t=z\}.$$
		The proof is completed by observing that uniformly in $k\le K$,
		$$\Pan_{unif}(B_{k+1}\:|\: B_k)=O\( \frac{Kt^2}{n}\)=O\(\frac{\log^7(n)}{n} \),$$
		which is sufficient to prove \eqref{eq:bound-an}.
		The above estimate simply follows by observing that $X^{(k+1)}_t=z$ implies that $X^{(k+1)}$ hits at some time $s\in[0,t]$ for the first time a vertex already discovered by the walks $ X^{(\ell)}$, $\ell\leq k$. 
	\end{proof}
	%\end{color}
Lemma \ref{le:widespread-pi} and Lemma \ref{le:max-pi} provide the result mentioned at the beginning of the section.
	\begin{corollary}\label{co:widespread-pi}
	W.h.p. $\pi_\si$ is a widespread measure.
	\end{corollary}
		\subsection{Proof of Theorem \ref{th:2cutoff}} We now turn to the proof of Theorem \ref{th:2cutoff}. Let $\sigma,\eta$ be two independent uniformly random configurations in $\cC$. In this section we will assume that $t=\Theta(\log n)$ and $s\le t$. Let $\mathbf{Q}^{\sigma,\eta}_x\equiv \mathbf{Q}^{\sigma,\eta}_{x,s,t}$ denote the quenched law of the walker that starts at $X_0=x$, goes for $s$ steps trough $\sigma$ and then, starting at $X_s$, goes for $t-s$ steps trough $\eta$. We use the notation
%		will represent the  probability distribution on $[n]$ for the position of the walker at time $t$ by the symbol
	\begin{equation}\label{eq:notas} Q_s^t(x,y)=\mathbf{Q}^{\sigma,\eta}_x(X_t=y)=Q^{s,t}_{\si,\eta}(x,y)=\sum_{z\in [n]}P_\sigma^s(x,z)P_\eta^{t-s}(z,y)
.\end{equation}

	\begin{comment}
	%	Fixed $\omega=\si,\eta$ and an arbitrary sequence of vertices $p=(v_0,\dots,v_\ell)$ we will use the notation $\w_\omega(p)$ to mean
	%	$$\w_\omega(p)=\prod_{j=0}^{\ell-1}P_\omega(v_j,v_{j+1}).$$
		We define \emph{path} of length $\ell$ a sequence of edges $\p= ((e_1,f_1),\dots,(e_\ell,f_\ell))$ of length $\ell$. We call \emph{weight} of the path $\p$ the product
		$$\w(\p)=\prod_{j=1}^{s\wedge \ell}\frac{1}{d^+_{v(e_j)}}\ind_{\sigma(e_j)=f_j}\prod_{j=s}^{\ell}\frac{1}{d^+_{v(e_j)}}\ind_{\eta(e_j)=f_j}.
		$$
	\end{comment}
		\begin{definition}
			We define \emph{path of length $t$} an arbitrary sequence of vertices $\p= (v_0,\dots,v_t)$.
		%	Fixed $\omega=\si,\eta$ and an arbitrary sequence of vertices $\p=(v_0,\dots,v_\ell)$ we will use the notation $\w_\omega(\p)$ to mean
		%	$$\w_\omega(\p)=\prod_{j=0}^{\ell-1}P_\omega(v_j,v_{j+1}).$$
			We call weight of the path $\w(\p)$ the product
			$$\w(\p)=\prod_{j=0}^{s-1}P_\si(v_j,v_{j+1})\prod_{i=s}^{t-1}P_\eta(v_i,v_{i+1}).%=\w_\si(\p')\w_\eta(\p'')
			$$
		\end{definition}
	\begin{lemma}\label{le:lln}
	%	If $s\in[0,t]$ and $t\in\Theta(\log n)$, for every $\varepsilon\in(0,1)$
	%	$$\min_{x\in[n]}\mathbf{Q}^{\sigma,\eta}_x\(\prod_{\ell=0}^{s-1}P_\si(X_\ell,X_{\ell+1})\prod_{\ell=s}^{t-1}P_\eta(X_\ell,X_{\ell+1})\in\[e^{-(1+\varepsilon)H t},e^{-(1-\varepsilon)H t}\] \)\overset{\P}{\longrightarrow}1.$$
		If  $t=t(n)=\Theta(\log n)$, %and $s=s(n)\in[0,t]$, 
		then for every $\varepsilon\in(0,1)$
		$$\min_{s\in[0,t]}\min_{x\in[n]}\mathbf{Q}^{\sigma,\eta}_x\(\w(X_0,X_1,\dots,X_t)\in\[e^{-(1+\varepsilon)H t},e^{-(1-\varepsilon)H t}\] \)\overset{\P}{\longrightarrow}1,$$
where the convergence holds in probability with respect to the independent pair $(\si,\eta)$. 	\end{lemma}
	\begin{proof}
 In \cite[Proposition 8]{BCS1} for the DCM and \cite[Theorem 4]{BCS2} for the OCM it is shown that when $s=0$, and $t=\Theta(\log n)$ one has that for every $\d>0$, the event 
 \begin{equation}\label{eq:unifs}
 \min_{x\in[n]}\mathbf{Q}^{\eta}_x\(\w(X_0,X_1,\dots,X_t)\in\[e^{-(1+\varepsilon)H t},e^{-(1-\varepsilon)H t}\] \)\geq 1-\d,
 \end{equation}
has probability at least $1- n^{-c}$, for some constant $c>0$, and we use the notation $\mathbf{Q}^{\eta}_x$ for the quenched law of the random walk in the single environment $\eta$. We refer to the proof of \cite[Theorem 4]{BCS2} for the details. 
 Since $\w$ is a product
 and the pair $(\si,\eta)$ is independent, we may apply \eqref{eq:unifs} separately for $(X_0,X_1,\dots,X_s)$ in the environment $\si$ and $(X_{s},X_{s+1},\dots,X_t)$ in the environment $\eta$, 
 to obtain
 \begin{equation}\label{eq:unifs2}
 \min_{x\in[n]}\mathbf{Q}^{\si,\eta}_x\(\w(X_0,X_1,\dots,X_t)\in\[e^{-(1+2\varepsilon)H t},e^{-(1-2\varepsilon)H t}\] \)\geq 1-2\d,
 \end{equation}
with probability at least $1- 2n^{-c}$, for some constant $c>0$, provided that both $s$ and $t-s$ are $\Theta(\log n)$.
Taking a union bound over $s$, we obtain that for every fixed constant $a\in(0,1)$,   w.h.p.\ with respect to $(\si,\eta)$
 \begin{equation}\label{eq:unifs2}
 \min_{s\in[a t,(1-a)t]}\min_{x\in[n]}\mathbf{Q}^{\si,\eta}_x\(\w(X_0,X_1,\dots,X_t)\in\[e^{-(1+2\varepsilon)H t},e^{-(1-2\varepsilon)H t}\] \)\geq 1-2\d.
 \end{equation}
 
It remains to consider the case $0\leq s\leq a t$ and the case $0\leq t-s\leq a t$.
	 Notice  that  any path of length $s$ has weight between $\D^{-s}$ and $1$. %This holds uniformly in $s\in[0,a t]$, where $a$ is a sufficiently small constant depending only on the maximum degree $\D$. 
	% We refer to \cite[Lemma 3.1]{CQ:cover} for a proof of this fact. 
	 %	 Concerning the values $s\in[0,a t]$, u
	 Using the fact that $\D^{-s}$ is in the interval $[e^{-a_1 t},1]$, for some $a_1>0$ such that $a_1\to 0$ when $a\to 0$, \eqref{eq:unifs} implies that w.h.p.\ with respect to $(\si,\eta)$ 
	  \begin{equation}\label{eq:unifs3}
 \min_{s\in[0,a t]}\min_{x\in[n]}\mathbf{Q}^{\si,\eta}_x\(\w(X_0,X_1,\dots,X_t)\in\[e^{-(1+\varepsilon')H t},e^{-(1-\varepsilon')H t}\] \)\geq 1-2\d,
 \end{equation}
where $\varepsilon'=\varepsilon+a_1$. Inverting the role of $s$ and $t-s$ the same estimate holds for $s\in[(1-a)t,t]$. 
Since $a$, and therefore $a_1$, is arbitrarily small, the bounds  \eqref{eq:unifs2} and \eqref{eq:unifs3} imply the desired claim. 
	\end{proof}
	
	\begin{proof}[Proof of the Lower Bound of Theorem \ref{th:2cutoff}.] Let $t=(1-\varepsilon) T_\ent$ for some $\varepsilon>0$. Fix any $x\in [n]$ and call $U^s_x$ the set of vertices $y$ such that 
	$$Q_s^t(x,y)\geq e^{-(1+\varepsilon)Ht}=n^{-1+\varepsilon^2}.$$ 
	%For any $a\geq 0$ call $A$ the set of $y$ such that the left hand side in \eqref{equao1} is at least $a$. 
	Let $\cP_{x,y,t}$ denote the set of paths $\p= (v_0,\dots,v_t)$ with  $v_0=x$ and $v_t=y$, so that 
	$$Q_{s}^t(x,y)=\sum_{\p\in\cP_{x,y,t}}\w(\p).$$ For any $\p\in\cP_{x,y,t}$, by definition of $\w(\p)$ we have 
$$
\w(\p)\leq Q_s^t(x,y).
$$
Therefore,  if $\w(\p)\geq n^{-1+\varepsilon^2}$ for some $\p\in\cP_{x,y,t}$ it follows that $y\in U^s_x$, and 
\begin{align}\label{equao2}
\mathbf{Q}^{\sigma,\eta}_x\(\w(X_0,X_1,\dots,X_t)\geq n^{-1+\varepsilon^2} \) &= \sum_{y\in[n]}\sum_{\p\in\cP_{x,y,t}} \w(\p)\ind(\w(\p)\geq n^{-1+\varepsilon^2})\nonumber
\\
&\leq 
\sum_{y\in U_x^s}Q_s^t(x,y).
\end{align}
It follows from Lemma \ref{le:lln}, that  %$|U_x|=o(n)$ and
	 for all $\d>0$, w.h.p.\ with respect to  $(\si,\eta)$
	 \begin{equation}\label{quxs}
	 \min_{s\in[0,t]}\min_{x\in[n]}Q_s^t(x,U^s_x)\geq 1-\d.
	 \end{equation}
	  Moreover,
	$$
	n^{-1+\varepsilon^2}|U^s_x|\leq \sum_{y\in U^s_x} Q_s^t(x,y)\leq 1,$$
	shows that  $\max_{s\in[0,t]}\max_{x\in[n]}|U^s_x|=o(n)$. Since $\pi_\eta$ is w.h.p.\ widespread by Corollary \ref{co:widespread-pi}, from the argument in \eqref{eq:csc} it follows that $\max_{s\in[0,t]}\max_{x\in[n]}\pi_\eta(U^s_x)=o(1)$. Hence, since $\d>0$ is arbitrarily small we conclude that w.h.p.\ uniformly in $x\in[n]$ and $s\in[0,t]$,  the measure $Q_s^t(x,\cdot)$ is asymptotically  singular with respect to $\pi_\eta$. This concludes the proof of \eqref{lowbobo}, the lower bound in Theorem \ref{th:2cutoff}.
	\end{proof}
%Another consequence of Lemma \ref{le:lln} is the following statement that was used in the proof of Theorem \ref{th:trichotomy-joint}, see Subsection \ref{subsec222}.
%
%	\begin{corollary}\label{le:llncor} 
%	Fixe $\e>0$, let $s,t$ be as in Lemma \ref{le:lln} and call $Y_{s,t}$ the set of $y\in[n]$ such that
%	\begin{equation}\label{equao1}
%\sum_{z\in[n]}P^{s}_\si(x,z)P^{t-s}_\eta(z,y)\ge e^{-(1+\epsilon)H t}.
%\end{equation}
%	Then, w.h.p.\ with respect to the independent pair $(\si,\eta)$,  \begin{equation}\label{equao3}
%	\sum_{y\in Y_{s,t}}\sum_{z\in[n]}P^{s}_\si(x,z)P^{t-s}_\eta(z,y)=1-o(1).
%	\end{equation}
%\end{corollary}
%\begin{proof}
%For any $a\geq 0$ call $A$ the set of $y$ such that the left hand side in \eqref{equao1} is at least $a$. Let $\cP_{x,y,t}$ denote the set of paths $\p= (v_0,\dots,v_t)$ with  $v_0=x$ and $v_t=y$. For any $\p\in\cP_{x,y,t}$ we have 
%$$
%\w(\p)\leq \sum_{z\in[n]}P^{s}_\si(x,z)P^{t-s}_\eta(z,y).
%$$
%Therefore,  if $\w(\p)\geq a$ it follows that $y\in A$, and 
%\begin{align}\label{equao2}
%\mathbf{Q}^{\sigma,\eta}_x\(\w(X_0,X_1,\dots,X_t)\geq a \) &= \sum_{y\in[n]}\sum_{\p\in\cP_{x,y,t}} \w(\p)\ind(\w(\p)\geq a)\nonumber
%\\
%&\leq 
%\sum_{y\in A}\sum_{z\in[n]}P^{s}_\si(x,z)P^{t-s}_\eta(z,y).
%\end{align}
%Taking $a=e^{-(1+\epsilon)H t}$, we have $A=Y_{s,t}$ and the claim \eqref{equao3} follows from Lemma \ref{le:lln}.
%\end{proof}

		We now turn to proving \eqref{upbobo}, the upper bound in Theorem \ref{th:2cutoff}, which is more involved. In fact, an adaptation of the arguments of \cite{BCS1,BCS2} is not straightforward in this case. Below we present the details in the case of the DCM only. The proof for the OCM is very similar.
	
	\begin{remark}\label{re:sologn}
		For what concerns the upper bound, we have $t=(1+a)\tent$ with $a=\beta-1 >0$ and  we want an estimate uniformly in $s\in I(\varepsilon,\beta)$. We consider the three regimes $s\in[0,\varepsilon\tent]$, $s\in[\varepsilon\tent, (1-\varepsilon)\tent]$ and $s\in[(1+\varepsilon)\tent, (\beta-\varepsilon)\tent]$. 
		%can restrict  to the case $s\in[a t,(1-a)t]$ for some sufficiently small $a>0$ 
		%because of the following argument. 
		In the first case, if $s\leq \varepsilon t$ then the upper bound in Theorem \ref{th:2cutoff} holds as a consequence of Theorem \ref{th:BCS}. Indeed, in this case we have $t-s\ge (1+a/2)T_\ent$ if $\varepsilon$ is small enough. Hence, by monotonicity of the total variation distance in a single environment  %uniformly in $s\leq \varepsilon t$,
		\begin{align*}
		\max_{s\in [0,\varepsilon t]}\| Q_s^t(x,\cdot)-\pi_\eta\|_{\tv}&\le \max_{s\in [0,\varepsilon t]}\max_{x\in[n]}\| P^{t-s}_\eta(x,\cdot)-\pi_\eta\|_{\tv}\\
		& \leq \max_{x\in[n]}\| P^{(1+a/2)T_\ent}_\eta(x,\cdot)-\pi_\eta\|_{\tv} \overset{\P}{\to}0.
		\end{align*}
		On the other hand, in the case $s\in[(1+\varepsilon)\tent, (\beta-\varepsilon)\tent]$, we write 		\begin{align*}
		\|Q_s^t(x,\cdot)-\pi_\eta \|_{\tv}\le& \|Q_s^t(x,\cdot)-\pi_\sigma P^{t-s}_\eta \|_{\tv}+\|\pi_\sigma P^{t-s}_\eta-\pi_\eta \|_{\tv}\\
		 \le&\|P_\si^s(x,\cdot)-\pi_\sigma \|_{\tv}+\|\pi_\sigma P^{t-s}_\eta-\pi_\eta \|_{\tv}\\ \le &\|P_\si^{(1+\varepsilon)\tent}(x,\cdot)-\pi_\sigma \|_{\tv}+\|\pi_\sigma P^{\varepsilon \tent}_\eta-\pi_\eta \|_{\tv}.
		\end{align*}
%		where we used  that 
By Theorem \ref{th:BCS} the first term in the right hand side tends to zero in probability, and the same holds for the second term by Corollary \ref{co:widespread-pi} and Lemma \ref{le:approx}.
% that $$\|\pi_\sigma P^{t-s}_\eta-\pi_\eta \|_{\tv}\overset{\P}{\to}0$$ by Corollary \ref{co:widespread-pi} and Lemma \ref{le:approx}, since $t-s\to\infty$. 
This shows that 	
			\begin{align*}
\max_{s[(1+\varepsilon)\tent, (\beta-\varepsilon)\tent]}	\max_{x\in[n]}	\| Q_s^t(x,\cdot)-\pi_\eta\|_{\tv}
\overset{\P}{\to}0,
\end{align*}
for all $\varepsilon\in(0,\min\{1,(\beta-1)/2\})$.

The rest of this section deals with the proof in the case $s\in[\varepsilon\tent, (1-\varepsilon)\tent]$. With some additional technical work it would be possible to treat the case $s\in[(1-\varepsilon)\tent,(1+\varepsilon)\tent]$ as well, so that the upper bound \eqref{upbobo} would hold uniformly in $s\in[0,(\beta-\varepsilon)\tent]$. However, since $\varepsilon$ is arbitrarily small, the statement \eqref{upbobo} is more than enough for our purposes (see \eqref{eqdefehat} and \eqref{eq:bigsmall}). 
	\end{remark}
	
	Thanks to Remark \ref{re:sologn} we can restrict to $s\in[\varepsilon\tent, (1-\varepsilon)\tent]$ for $\varepsilon>0 $ fixed and small enough. In what follow it will be convenient to actually assume that $$t=(1+\upsilon) \tent, \qquad 2\upsilon \tent\leq s\leq (1-2\upsilon)\tent,$$ for some sufficiently small $\upsilon>0$. This causes no loss of generality since the general case $t=\beta T_\ent$ for any $\beta>1$ then follows by monotonicity of the distance $\| Q_s^t(x,\cdot)-\pi_\eta\|_{\tv}$ with respect to $t$ (and $s$ fixed). Define
%	$$\hslash:=\frac{1}{5}\log_\D(n),\qquad h:=\hslash \wedge \frac{t-s}{2}=\Theta (T_\ent),\qquad r:=t-s-h=\Theta(T_\ent),$$
$$h=2\upsilon\tent,\qquad r=t-s-h=\Theta(T_\ent),$$
	and notice that %there exists some $\epsilon\in(0,1)$ such that
	$$r+s=t-h=(1-\upsilon) T_\ent.$$
	\subsection*{Strategy of proof} 
	The overall strategy of proof is similar to the one used in \cite{BCS1,BCS2} to prove the upper bound for Theorem \ref{th:BCS}, but the presence of a double environment and the need for uniformity in $s$ require some additional technical steps. We first recall the main ideas from \cite{BCS1,BCS2} and then give the details of their implementation in our more general setting.
	We can replace $\pi_\eta$ by $\muin P_\eta^h$ since we know by Lemma \ref{le:approx} that w.h.p.
		\begin{equation}
		\| \muin P_\eta^h-\pi_\eta\|_\tv =o(1).
		\end{equation}
		We will focus on a particular set of starting states. %Fixed the realization $\sigma$ w
		We call
		$S^\si_\star$ the set of vertices for which the out-neighborhood in $G(\si)$ is a tree up to height $h$.
		By \cite[Proposition 6]{BCS1} (or \cite[Lemma 9]{BCS2}) w.h.p. with respect to the configuration $\si$, most of the vertices are in $S^\si_\star$, and the quenched probability that the walk is out of the set $S_\star^\si$ vanishes exponentially fast in time. More precisely, if $\ell$ is any integer $\ell\leq \varepsilon\tent$, and $\varepsilon$ is a sufficiently small constant, then w.h.p.\ for all $s\in[\varepsilon\tent, (1-\varepsilon)\tent]$
		\begin{equation}\label{eq:ell}
		\max_{x\in[n]}\mathbf{Q}_x^{\si,\eta}\(X_\ell\not\in S_\star^\si \)\le 2^{-\ell}.
		\end{equation}
		Therefore, considering the event $\{X_\ell\in S_\star^\si\}$ and its complement one has 
		\begin{equation}\label{eq:ub4}
		\max_{x\in[n]}\| Q_s^t(x,\cdot)-\muin P^h_\eta \|_{\tv}\le \max_{x\in[n]}\mathbf{Q}^{\si,\eta}_x(X_\ell\not\in S_\star^\sigma)+\max_{x\in S_\star^\sigma}\| Q_{s-\ell}^{t-\ell}(x,\cdot)-\muin P^h_\eta\|_{\tv}.
		\end{equation}
		Thus, in order to show the uniform upper bound in Theorem \ref{th:2cutoff} it is sufficient to show an upper bound that holds uniformly in the random set $S_\star^\si$.

		Below we will define a set of \emph{nice paths} for the trajectory of the walk, see Definition \ref{def:nice}. For every pair of vertices $x,y$ we let $\cN_{x,y}$ denote the set of nice paths from $x$ to $y$ of length $t$. Consequently, we define
		$$\bar Q_{s}^t(x,y):=\sum_{\p\in\cN_{x,y}}\w(\p)$$
		the probability to go from $x$ to $y$ along a nice path. Notice that for any
		 %arbitrarily small constant 
		 $\varepsilon>0$,
			\begin{align}
		\nonumber	\|\muin P_\eta^h-Q_s^t(x,\cdot)\|_\tv=&\sum_{y\in[n]}\[\muin P_\eta^h(y)-Q_s^t(x,y) \]_+\\
			\label{eq:ub1}\le& \sum_{y\in[n]}\[\muin P_\eta^h(y)(1+\varepsilon)+\frac{\varepsilon}{n}-\bar Q_{s}^t(x,y)\]_+,
			\end{align}
			where we have used $Q_s^t(x,y)\geq \bar Q_{s}^t(x,y)$. 
			Therefore,	if we can show that
			\begin{equation}\label{eq:validity}
			\muin P_\eta^h(y)(1+\varepsilon)+\frac{\varepsilon}{n}\ge \bar Q_{s}^t(x,y),
			\end{equation}
		then the positive part in \eqref{eq:ub1} can be neglected, and summing over $y\in[n]$ one obtains
			\begin{align}\label{eq:ub2}
			\|\muin P_\eta^h-Q_s^t(x,\cdot)\|_\tv&\le \sum_{y\in[n]}\((1+\varepsilon)\muin P_\eta^h(y)+\frac{\varepsilon}{n}-\bar Q_{s}^t(x,y)\)\\
		\nonumber	&= 2\varepsilon+\mathbf{Q}^{\sigma,\eta}_x\Big( (X_0,\dots,X_t)\not\in\cup_{y\in[n]}\cN_{x,y}\Big).
			\end{align}
			At this point we are left with showing that the probability of following a path that is not nice is arbitrarily small uniformly in the starting point $x\in S_\star^\si$, namely
			\begin{equation}\label{eq:nicepaths}
			\max_{x\in S^\si_\star}\mathbf{Q}_x^{\sigma,\eta}\Big( (X_0,\dots,X_t)\not\in\cup_{y\in[n]}\cN_{x,y}\Big)<\varepsilon. %\qquad \text{w.h.p.}
			\end{equation}
We first introduce the notation required to define the set of nice paths. Then we will present a proof of the validity of  \eqref{eq:validity} and \eqref{eq:nicepaths}. 
More precisely, we will show that w.h.p.\ with respect to the independent pair $(\si,\eta)$, one has that both  \eqref{eq:validity} and \eqref{eq:nicepaths}  hold for every $s\in[2\upsilon \tent,(1-2\upsilon)\tent]$.

We start by constructing the subgraph $\cG_x^\si(s)$ of $G(\si)$ spanned by the paths of length at most $s$, starting at $x$, and with weight at least $e^{-(1+\upsilon)Hs}$. We  construct %such subgraph 
$\cG_x^\si(s)$ together with a spanning tree $\cT_x^\si(s)$ of  $\cG_x^\si(s)$ in the following way. 
	\begin{definition}\label{def:construction1}\textbf{Construction of $\cG_x^\si(s)$ and $\cT_x^\si(s)$.}
			\begin{itemize}
				\item Call $\cG^\sigma[0]$ the empty graph on $\{x \}$ and  $E^\si_1=E^+_x$.
				\item To every $e_1\in E_1^\si$ associate the weight $\widehat \w_\si(e_1):=(d_{x}^+)^{-1}$.
				\item
				Recursively, for every $\ell\ge 1$:
				\begin{enumerate}
					\item Choose a tail $e_\ell\in E^\si_\ell$ with maximal weight and reveal $\si(e_\ell)=f_\ell$. 
					\item Add the edge $(e_\ell,f_\ell)$ to $\cG^\si[\ell-1]$ and call the resulting graph $\cG^\si[\ell]$.
					\item Call the edge $(e_\ell,f_\ell)$ \emph{open} if $v(f_\ell)\not\in\cG^\si[\ell-1]$.
					\item Call $\cT^{\si}[\ell]$ the \emph{open} subgraph of $\cG^\si[\ell]$.
						\item If $v(f_\ell)\not\in\cG^\si[\ell-1]$, then  associate to any $e'\in E^+_{v(f_\ell)}$ the weight $\widehat{\w}_\si(e'):=\widehat{\w}_\si(e_\ell)(d^+_{v(f_\ell)})^{-1}$, and if $$\widehat{\w}_\si(e')\ge e^{-(1+\upsilon)Hs},$$ 	then let ${E}^\si_{\ell+1}=E^\si_\ell\setminus\{e_\ell \}\cup E^+_{v(f_\ell)}$.  Otherwise, set ${E}^\si_{\ell+1}=E^\si_\ell\setminus\{e_\ell \}$.
					\item Remove from $E^\si_{\ell+1}$ the tails $e'$ such that the vertex $v(e')$ is at distance greater than $s$ from $x$ in $\cT^\si[\ell]$.
				\end{enumerate}
			\item Iterate the instructions above up to the random time $\kappa_\si$ at which $E^\si_{\kappa_\si}=\emptyset$, and call
			$$\cT_x^\si(s):=\cT^\si[\kappa_\si],\qquad \cG_x^\si(s):=\cG^\si[\kappa_\si].$$
			\end{itemize}
		\end{definition}
\begin{remark}\label{rem:co}
		The definition given above of the subgraphs $\cT_x^\si(s)$ and $\cG_x^\si(s)$ coincides with that given in \cite{BCS1,BCS2}. It was shown in   \cite[Proposition 13]{BCS2} and \cite[Proposition 10]{BCS1} that the random walk on the static environment $\si$, starting at $x\in S_\star^\si$ and of length $s_{\max}=(1-2\upsilon)\tent$ will stay on the tree $\cT_x^\si(s_{\max})$ w.h.p. In particular,  in the double environment setup, w.h.p.\  for every $s\in[0,(1-2\upsilon)\tent]$  the walk will be in one of the leaves of $\cT_x^\si(s)$ at time $s$ with probability at least $1-\d$ for each fixed $\d>0$. 
\end{remark}		
		 Call $\cL^{x,\sigma}_s$ the set of leaves of $\cT^\sigma_x(s)$ at distance $s$ from $x$. We  now construct the subgraph $\cG^{\si,\eta}_x(r)$ of $G(\eta)$ consisting of all the paths in $G(\eta)$ which start at some $z\in\cL^{x,\si}_s$, have length $r$, and weight larger than $e^{-(1+\upsilon/2)H(t-h)}$. Similarly to the construction in Definition \ref{def:construction1}, together with $\cG^{\si,\eta}_x(r)$ we are going to construct a collection $\cW_x^{\si,\eta}(r)$ of disjoint rooted directed trees, each rooted at some $z\in\cL_s^{x,\si}$ and with depth $r$. The forest $\cW_x^{\si,\eta}(r)$,
%		$$\cW_x^{\si,\eta}(r)=\bigcup_{z\in\cL_s^{x,\si}}\cT^{\si,\eta}_z(r)$$
		seen as a collection of edges, will be our candidate for the support of the walk from time $s$ to time $t-h$.
		\begin{definition}\label{def:construction2}\textbf{Construction of $\cG^{\si,\eta}_x(r)$ and $\cW_x^{\si,\eta}(r)$.}
		\begin{itemize}
			\item Call $\cG^{\si,\eta}[0]$ the empty graph on $\cL_s^{x,\si}$
			and call  $E_1^{\si,\eta}=\cup_{z\in\cL^{x,\si}_s}E^+_z$.
			\item To every $e_1\in E_1^{\si,\eta}$  associate the weight 
			$$\widehat{\w}_{\si,\eta}(e_1):=\widehat \w_\si(e_1),$$ of the unique path in $\cT^\si_x(s)$ joining $x$ to $v(e_1)$ times the inverse of the out degree of $v(e_1)$; see Definition \ref{def:construction1}.
		%	\MQ{Li vogliamo chiamare $\widehat{\w}_\eta(e_1)$?}
			\item
			Recursively, for every $\ell\ge 1$
			\begin{enumerate}
				\item Choose a tail in $e_\ell\in E_\ell^{\si,\eta}$ with maximal weight and reveal $\eta(e_\ell)=f_\ell$. 
				
				\item Add the edge $(e_\ell,f_\ell)$ to $\cG^{\si,\eta}[\ell-1]$ and call the resulting graph $\cG^{\si,\eta}[\ell]$.
					\item Call the edge $(e_\ell,f_\ell)$ \emph{open} if $v(f_\ell)\not\in\cG^{\si,\eta}[\ell-1]$.

					\item Call $\cW^{\si,\eta}[\ell]$ the \emph{open} subgraph of $\cG^{\si,\eta}[\ell]$.

\item If $v(f_\ell)\not\in\cG^{\si,\eta}[\ell-1]$, then associate to any $e'\in E^+_{v(f_\ell)}$ the weight $\widehat{\w}_{\si,\eta}(e'):=\widehat{\w}_{\si,\eta}(e_\ell)(d^+_{v(f_\ell)})^{-1}$, and if $$\widehat{\w}_{\si,\eta}(e')\ge e^{-(1+\upsilon/2)H(t-h)}=:\widehat{\w}_{\min},$$ 	then set ${E}^{\si,\eta}_{\ell+1}=E^{\si,\eta}_\ell\setminus\{e_\ell \}\cup E^+_{v(f_\ell)}$.  Otherwise, set ${E}^{\si,\eta}_{\ell+1}=E^{\si,\eta}_\ell\setminus\{e_\ell \}$.
					\item Remove from $ E^{\si,\eta}_{\ell+1}$ the tails $e'$ such that the vertex $v(e')$ is at distance greater than $r$ from the corresponding root in $\cW^{\si,\eta}[\ell]$. 
			\end{enumerate}
			\item Iterate the instructions above up to the random time $\kappa_{\si,\eta}$ at which $E^{\si,\eta}_{\kappa_{\si,\eta}}=\emptyset$, and call
			$$%\cT_z^{\si,\eta}(r):=\cT_{z}^{\si,\eta}[\kappa_\eta],\qquad 
			\cW_x^{\si,\eta}(r):=\cW^{\si,\eta}[\kappa_{\si,\eta}],
			%=\cup_{z\in\cL^{\si,x}_s}\cT_{z}^{\si,\eta}[\kappa_{\si,\eta}]
			\qquad \cG^{\si,\eta}_x(r):=\cG^{\si,\eta}[\kappa_{\si,\eta}].$$
			% We call the leaves of the tree rooted at $z$, $\cL^{z,\eta}_r$.
		\end{itemize}
\end{definition}
As in \cite{BCS1,BCS2} one has that the random number of edges revealed by the construction in Definition \ref{def:construction1}, $\kappa_\si$, is a.s.\ $o(n)$. We need an analogous result for the quantity $\kappa_{\si,\eta}$ in Definition \ref{def:construction2}.
	\begin{lemma}\label{le:lemma2}
		For any $\sigma,\eta\in\cC$ and $x\in [n]$, %for all $\eta\in\cC$
		$$\widehat{\w}_{\si,\eta}(e_\ell)\le\frac{r}{\ell},\qquad\forall \ell< \kappa_{\si,\eta}.$$
		In particular, recalling that $t-h=r+s= (1-\upsilon)\tent$, %by choosing $\upsilon\le \epsilon$
		$$\kappa_{\si,\eta} = %\le1+  re^{(1+\upsilon)H(s+r)}=
		O\(\log(n)n^{(1+\upsilon/2)(1-\upsilon)} \)=O(n^{1-\upsilon^2}).$$
	\end{lemma}
	\begin{proof}
		For each $\ell<\kappa_{\si,\eta}$ we consider the forest $\widetilde{\cW}^{\si,\eta}[\ell]$ %:=\cup_{z\in\cL^{x,\si}_s}\widetilde{\cT}_{z}^{\si,\eta}[\ell]$$
		constructed as in Definition \ref{def:construction2}, but if an edge $(e_{\ell'},f_{\ell'})$ for some $\ell'\le \ell$ is not open, we attach a fictitious leaf (with no future children) to $e_{\ell'}$, to which we assign the weight $\widehat{\w}_{\si,\eta}(e_{\ell'})$. This construction ensures that for every $\ell$ both the graph $\cG^{\si,\eta}[\ell]$ and the forest $\widetilde\cW^{\si,\eta}[\ell]$  have exactly $\ell$ edges. Call $F_\ell$ the set of leaves of $\widetilde\cW^{\si,\eta}[\ell]$.
		%, both the leafs of $\cW^{\si,\eta}[\ell]$ and the fictitious ones. 
		%Consider the set of paths $\p:z\to v$ in $\widetilde{W}^{\si,\eta}[\ell]$ for some $z\in\cL_s^{x,\si} $ and $v\in F_\ell$. 
		By construction, for all  $v\in F_\ell$ there is a unique $z\in\cL_s^{x,\si}$ and a unique path $\p(v):z\to v$ of length at most $r$ in $\widetilde{W}^{\si,\eta}[\ell]$. The weight of such a path is given by $\widehat{\w}_{\si,\eta}(e_v)$ where $e_v$ is any tail in $E^+_{u}$ if $(u,v)$ is the last edge in the path $\p(v)$. It follows that 
%		 denote the 
%		the weight of such a path is $\w(\p)=\widehat\w(e)$ where $e\in E^+_v$.
%		With this notation, we have
		$$	\sum_{z\in\cL_s^{x,\si}}\sum_{v\in F_\ell}\sum_{\p:z\to v}\widehat{\w}_{\si,\eta}(e_v)\le 1.$$
		Since all $v\in F_\ell$ are such that $\widehat{\w}_{\si,\eta}(e_v)\geq \widehat{\w}_{\si,\eta}(e_\ell)$, 
		% for $e\in E^+_v$, 
		%Notice that for every  $v\in F_\ell$ there there exist a unique $z\in\cL_s^{x,\si}$ admitting a (unique) path $\p(z,v)$ in $\widetilde\cW^{\si,\eta}[\ell]$ joining $z$ to $v$. Moreover, such path must be of length at most $r$. 
		% and by definition  $\widehat{\w}(e_{\ell+1})\le \widehat{\w}(e_{\ell})$, hence if there is a path $\p:z\to v$ it must be the case that $\w(\p)\ge \widehat{\w}(e_{\ell})$. 
we obtain		$$|F_\ell|\widehat{\w}_{\si,\eta}(e_{\ell})\le1.$$
		By the absence of cycles in $\widetilde\cW^{\si,\eta}[\ell]$ we also have that
		$$\ell\le r|F_\ell|.$$
		In conclusion
		$$\widehat{\w}_{\si,\eta}(e_{\ell})\le\frac{1}{|F_\ell|}\le \frac{r}{\ell}.$$
		If we replace $\ell=\kappa_{\si,\eta}-1$ we get
		$$\kappa_{\si,\eta} -1\le \frac{r}{\widehat{\w}_{\si,\eta}(e_{\kappa_{\si,\eta}-1})}\le \frac{r}{\widehat{\w}_{\min}}. \qedhere$$
		\end{proof}
	Next, we define the set of nice paths.
		\begin{definition}\label{def:nice}
			We call \emph{nice} a path $\p=(v_0,\dots,v_t)$ s.t.
			\begin{enumerate}
				\item $%e^{-(1+\upsilon/2)Ht}<
				\w(\p)\le e^{-(1-\upsilon/2)Ht}$.
				%=n^{-1-\upsilon/2-\upsilon^2/2}=O(n^{-1-\upsilon/3})$.
				\item $\p$ belongs to $\cT_{v_0}^\si(s)$ up to time $s$.
				\item $\p$ belongs to $\cW_{x}^{\si,\eta}(r)$ from time $s$ to time $t-h=r+s$.
				% for some $z\in\cL_{s}^{v_0,\si}$.
				\item $(v_{t-h},\dots,v_t)$ is the unique path of length at most $h$ in the graph $G(\eta)$ from $v_{t-h}$ to $v_t$.
			\end{enumerate}
			For every $x,y\in[n]$, we write $\cN_{x,y}$ for the set of nice paths $(v_0,\dots,v_t)$ with $v_0=x$ and $v_t=y$. 
		\end{definition}
We now focus on proving \eqref{eq:nicepaths}, which will be a consequence of the law of large numbers in Lemma \ref{le:lln} together with the forthcoming Lemma \ref{le:lemma1}. The latter shows, via a martingale argument, that w.h.p.\ the law of the walk is concentrated on trajectories that do  not exit %along the edges of one of the trees of 
the forest $\cW_x^{\si,\eta}(r)$ during the time steps $s,\dots,s+r$.

Fix $\sigma\in\cC$ and $x\in [n]$. The set of leaves $\cL^{\si,x}_s$ is then determined, %therefore a measurable quantity 
and we call $\P_{\si,x}$ the law of the process defined  in Definition \ref{def:construction2}. Consider the $\sigma$-field $(\cS_\ell)_{\ell\ge 0}$ generated by the first $\ell$ steps of the construction 
%of 
%$\eta$ as 
described in Definition \ref{def:construction2}. Call $(M_\ell)_{\ell\ge0}$ the stochastic process adapted to $(\cS_\ell)_{\ell\ge 0}$ defined recursively by $M_0=0$ and 
$$M_{\ell+1}=M_\ell+\ind_{\ell+1<\kappa_{\si,\eta}}\ind_{v(f_{\ell+1})\in\cG^{\si,\eta}[\ell]}\widehat{\w}_{\si,\eta}{(e_{\ell+1})}.$$
	\begin{lemma}\label{le:lemma1}
		For every $\varepsilon>0$, $x\in[n]$, for all $\si\in\cC$ and $s\in[2\upsilon \tent,(1-2\upsilon)\tent]$,
		$$\P_{\si,x}\(M_{\kappa_{\si,\eta}}\le\varepsilon \)\geq 1-n^{-3/2},$$
	for all $n$ large enough.
	\end{lemma}
	
	\begin{proof}
		We follow \cite{BCS1,BCS2}, where a very similar statement was proved for the walk on a single environment.  For simplicity we write $\widehat{\w}$ instead of $\widehat{\w}_{\si,\eta}$. We compute the first two conditional moments of the increment $M_{\ell+1}-M_{\ell}$:
		\begin{align}
		\E\[M_{\ell+1}-M_{\ell}\:|\:\cS_\ell\]&\le\ind_{\ell+1< \kappa_{\si,\eta}}\frac{\widehat{\w}(e_{\ell+1})\D|\cG^{\si,\eta}[\ell]|}{m-\ell},\\
		\E\[(M_{\ell+1}-M_{\ell})^2\:|\:\cS_\ell\]&\le\ind_{\ell+1< \kappa_{\si,\eta}}\frac{\widehat{\w}(e_{\ell+1})^2\D|\cG^{\si,\eta}[\ell]|}{m-\ell}.
		\end{align}
		Fix any $\bar{\ell}=\Theta(\log n)$ and observe that since $|\cG^{\si,\eta}[\ell]|\le\ell$, $\widehat{\w}(e_\ell)\le\frac{r}{\ell}$, we have
		$$
		%|\cG^{\si,\eta}[\ell]|\le\ell,\quad \widehat{\w}(e_\ell)\le\frac{r}{\ell}\qquad\Longrightarrow\qquad 
		\widehat{\w}(e_{\ell+1})\D|\cG^{\si,\eta}[\ell]|=O(\log n),\quad\sum_{\ell\ge\bar{\ell}}\widehat{\w}(e_{\ell})=O(\log^2(n)).$$
	 Set
		\begin{gather*}a:=\sum_{\ell\ge\bar{\ell}}\E\[M_{\ell+1}-M_{\ell}\:|\:\cS_\ell\]=O\(\log(n)n^{-\epsilon^2}\)=o(1),\\
		b:=\sum_{\ell\ge\bar{\ell}}\E\[(M_{\ell+1}-M_{\ell})^2\:|\:\cS_\ell\]=O\(\log^3(n)n^{-1} \).\end{gather*}
Fix any $\varepsilon\in(0,1)$ and define
	$$Z_{\ell+1}=\frac{4}{\varepsilon}\( M_{\ell+1}-M_\ell - \E\[M_{\ell+1}-M_{\ell}\:|\:\cS_\ell\]\).$$
We observe that
	$\E[Z_{\ell+1}\:|\:\cS_\ell ]=0$ and that 
		$|Z_{\ell+1}|\le 1$, since if $\ell\ge\bar{\ell}\gg 1$, then $\widehat\w(e_{\ell+1})\to 0$, and in particular $M_{\ell+1}-M_\ell\le \frac{\varepsilon}{4}$.
	Consider the martingale
	$$W_u=\sum_{\ell=\bar{\ell}+1}^{u}Z_\ell,\qquad\forall u>\bar\ell.$$
Notice that
	$$W_{\kappa_{\si,\eta}}=\frac{4}{\e}\(M_{\kappa_{\si,\eta}}-M_{\bar{\ell}}-a \)\qquad\text{and}\qquad b':=\sum_{\ell\ge\bar{\ell}}\var(Z_\ell\:|\:\cS_\ell) \le\frac{16}{\varepsilon^2}b.$$
	A martingale version of Bennett's inequality introduced in \cite[Theorem 1.6]{FR75} ensures that, for $c>0$,
		$$\P_{\si,x}\(\exists u\ge\bar{\ell}\text{ s.t. }W_u\ge c \)\le e^c\(\frac{b'}{c+b'}\)^{c+b'}.$$
		In particular,
		\begin{align*}
		\P_{\si,x}\(M_{\kappa_{\si,\eta}}-M_{\bar{\ell}}\ge \e \)&=\P_{\si,x}\(\frac{\e}{4}W_{\kappa_{\si,\eta}}+a \ge \varepsilon \)\\&
		\le\P_{\si,x}\(\frac{\e}{4}W_{\kappa_{\si,\eta}}\ge \frac{\varepsilon}{2} \)=\P_{\si,x}\(W_{\kappa_{\si,\eta}}\ge 2 \)=o(n^{-3/2}).
		\end{align*}
		We are left to show that for every $\varepsilon>0$,
		 $$\P_{\sigma,x}\(M_{\bar{\ell}}\le \varepsilon \)\geq1-n^{-3/2}.$$
		 	The number of non-open edges in the first $\bar\ell$ steps in the process from Definition \ref{def:construction2} is stochastically dominated by a binomial with parameters $\bar \ell$ and $$p=\D\bar\ell/(m-\bar \ell) = O(n^{-1}\log n).$$ Therefore, the probability of having $2$ or more edges that are not open in the first $\bar\ell$ steps is $o(n^{-3/2})$. Combined with the fact that 
			%Moreover, it trivially holds that 
			for every $\ell\ge 0$, $\widehat{\w}(e_{\ell})\le 2^{-s}$, this implies 
		 	$$\P_{\sigma,x}(M_{\bar{\ell}}\ge 2^{-s+1})\leq n^{-3/2},$$
		 	which is enough to derive the desired conclusions.
	\end{proof}
The proof of \eqref{eq:nicepaths} is achieved by collecting the results obtained so far.
\begin{proposition}\label{pr:nice}
	For every $\varepsilon>0$:
	$$\lim_{n\to\infty}\P\Big(\min_{s\in[2\upsilon \tent,(1-2\upsilon)\tent]}\min_{x\in S_\star^\si} \mathbf{Q}^{\sigma,\eta}_x\Big((X_0,\dots,X_{t})\in\cup_{y\in[n]}\cN_{x,y}\Big)>1-\varepsilon \Big)=1.$$
\end{proposition}
	\begin{proof}%[Proof of Proposition \ref{pr:nice}]
		%\begin{equation}
		%\mathbf{Q}^{\sigma,\eta}_x((X_0,\dots,X_t)\in\cN_x)=\sum_{y\in[n]}Q_{0,t}(x,y).
		%\end{equation}
		%Most of the proof is a direct consequence of the results in \cite{BCS1,BCS2}. Nonetheless, in point (3) it is needed an extra argument that allows to \emph{glue} the two environments. 
		We check the conditions in  Definition \ref{def:nice} one by one:
	\begin{enumerate}
		\item follows from  Lemma \ref{le:lln};
		\item  this follows from \cite[Proposition 13]{BCS2} and \cite[Proposition 10]{BCS1}, see Remark \ref{rem:co}.
				\item for a fixed $s$, the third requirement in Definition \ref{def:nice} follows from  Lemma \ref{le:lemma1}. Indeed, $M_{\kappa'_{\eta}}$ is greater or equal than the probability that a random walk starting at $x$ in $\si$ and visiting $\cL^{\si,x}_s$ at time $s$, exits the forest $\cW_x^{\si,\eta}(r)$ in the time interval $[s,t-h]$. The uniformity in $s\in[2\upsilon \tent,(1-2\upsilon)\tent]$ follows by taking a union bound over $s$ in  Lemma \ref{le:lemma1}.
		%a moving trough the environment $\eta$, traverses an edge which is not in $\cW_x^{\si,\eta}(r)$.
		\item in order to satisfy the fourth requirement of Definition \ref{def:nice} it is sufficient that $v_{t-h}\in S_\star^\eta$. Therefore we obtain the desired conclusion by noticing that w.h.p.\
		$$\max_{z\in[n]} \mathbf{P}^\eta_z(X_r\not\in S_\star^\eta)\leq 2^{-\upsilon\tent},$$ 
		uniformly in $s\in[2\upsilon\tent,(1-2\upsilon)\tent]$, where we use  $r\geq \upsilon\tent$ and the bound \eqref{eq:ell}.
		% see \cite[Proposition 6]{BCS1} and \cite[Lemma 9]{BCS2}.
	\qedhere\end{enumerate}
	\end{proof}

	We are now left with showing the validity of \eqref{eq:validity}. Such a result is achieved by the following lemma, which is based on the constructions in Definitions \ref{def:construction1} and \ref{def:construction2} and on a concentration inequality.
	\begin{lemma}\label{le:chat2env}
	For every $\varepsilon>0$
		$$\lim_{n\to\infty}\P\(\forall x,y\in[n],\:\max_{s\in[2\upsilon \tent,(1-2\upsilon)\tent]}\bar Q_{s}^t(x,y)\leq (1+\varepsilon)\muin P_\eta^h(y)+\tfrac{\varepsilon}{n} \)=1.$$
	\end{lemma}
	\begin{proof}%[Proof of Lemma \ref{le:chat2env}.]
		Fix $x,y\in[n]$ and $s\in[2\upsilon \tent,(1-2\upsilon)\tent]$. Generate, in this order, the configuration $\si$, the graph $\cG^{\si,\eta}_x(r)$ as in Definition \ref{def:construction2}, and the in-neighborhood of $y$ up to distance $h-1$. The latter can be constructed in the usual \emph{breadth first} way, see e.g.\ \cite[Section 5.3]{CQ:PageRank}. Let $\cS$ denote the $\sigma$-field generated by this construction. Clearly, in the construction of the in-neighborhood of $y$ we cannot reveal more than $\Delta^h=o(n)$ edges. Therefore, by Lemma \ref{le:lemma2} at most $o(n)$ edges of $\eta$ have been revealed up to this point. 
		Let $\cE_\cW$ denote the tails of the leaves of $\cW_x^{\si,\eta}(r)$ at distance $r$ from $\cL_s^{x,\si}$ and call $\cF_y$ the set of heads of the vertices $v$ in the boundary of the in-neighborhood of $y$ such that there is a unique path of length at most $h-1$ to $y$ in the configuration $\eta$. Both $\cE_\cW$ and $\cF_y$ are $\cS$-measurable, and 
		\begin{align}\label{eq:efm}
		\E[\ind_{\eta(e)=f}\:|\:\cS]=\frac{1}{m}(1+o(1)),\end{align}
		 for any $e\in\cE_\cW$, $f\in\cF_y$.
		Associate to each head $f\in\cF_y$ the weight $$\widehat\w'(f)=P^{h-1}_\eta(v(f),y).$$ At this point we notice that by definition of nice paths, 
			\begin{align*}
	\bar Q^t_s(x,y)=&\sum_{e\in\cE_\cW}\widehat{\w}_{\si,\eta}(e)\sum_{f\in \cF_y}\widehat\w'(f)\ind_{\widehat{\w}(e)\widehat{\w}'(f)\le e^{-(1-\upsilon/2)Ht}}\ind_{\eta(e)=f}.
	\end{align*}
	We remark that 
	\begin{align}\label{eq:efm1}
	\frac1m\sum_{f\in \cF_y}\widehat\w'(f)\le \muin P_\eta^h(y),\,\qquad \sum_{e\in\cE_\cW}\widehat{\w}_{\si,\eta}(e)\le 1.\end{align}
Since a matching $\eta(e)=f$ of $e\in \cE_\cW$ and $f\in\cF_y$ can only occur after the generation of $\si,\cG^{\si,\eta}_x(r),\cF_y$, \eqref{eq:efm} and \eqref{eq:efm1} show that
	$$\E[\bar Q^t_s(x,y)\:|\:\cS]\le \muin P_\eta^h(y).$$
	We rewrite
	$Z:=\bar Q^t_s(x,y)=\sum_{e\in\cE_\cW}c(e,\eta(e)),$
	where
	$$c(e,f)=\widehat{\w}_{\si,\eta}(e)\widehat\w'(f)\ind_{\widehat{\w}(e)\widehat{\w}'(f)\le e^{-(1-\upsilon/2)Ht}}.$$
	Here we observe that we can take $\upsilon$ such  that $(1-\upsilon/2)(1+\upsilon)\geq 1+\upsilon/3$ and therefore
	\begin{align}\label{eq:efm2}
	\|c\|_\infty=\max_{e,f}c(e,f)\leq e^{-(1-\upsilon/2)Ht}\le n^{-1-\upsilon/3}.
	\end{align}	
	We can now invoke the concentration inequality  (see \cite[Proposition 1.1]{chatterjee2007stein} and \cite[Section 6.2]{BCS1})
	$$\P\(Z-\E[Z|\cS]\ge a\:|\:\cS \)\le \exp\(-\frac{a^2}{2\| c\|_\infty(2\E[Z|\cS]+a)} \). $$
	Choosing $a:=\frac{\varepsilon}{2}\E[Z|\cS]+\frac{\varepsilon}{n}$, \eqref{eq:efm2} shows that this  probability is bounded by $o(n^{-3})$ for every fixed choice of $x,y$ and ${s\in[2\upsilon \tent,(1-2\upsilon)\tent]}$.  
	 Taking a union bound we conclude the desired result.
	\end{proof}

 	%----------------------------------------------------------------------
 	%% Bibliography
 	%----------------------------------------------------------------------
 	
 	\bigskip

 	\subsection*{Acknowledgments}
 	We acknowledge support of PRIN 2015 5PAWZB ``Large Scale Random Structures", and of INdAM-GNAMPA Project 2019 ``Markov chains and games on networks''.
 	%\newpage
 	\bigskip
 	
 	\bibliographystyle{plain}
 	 \bibliography{refresh200920.bib}
 	
 \end{document}